\documentclass[12pt,reqno]{amsart}
  
\usepackage{latexsym}
\usepackage{amssymb, tikz,bm, mathrsfs}

\usepackage[toc,page]{appendix}

\usepackage{color}

\usepackage{amsthm}
\theoremstyle{plain}

\newtheorem*{theorem*}{Theorem}

\renewcommand{\epsilon}{\varepsilon}

\newcommand{\N}{{\mathbb N}}
\newcommand{\R}{{\mathbb R}}

\newcommand{\Op}{{\rm{Op}}}
\newcommand{\Oph}{{\rm{Op}_{\h}}}

\newcommand{\h}{\hbar}

\newcommand{\supp}{{\operatorname{Supp\,}}}
\renewcommand{\phi}{\varphi}

\newcommand{\onehalf}{\frac{1}{2}}
\numberwithin{equation}{subsection}
\newtheorem{theo}[equation]{{\sc Theorem}}

\newtheorem{cor}[equation]{{\sc Corollary}}

\newtheorem{lem}[equation]{{\sc Lemma}}

\newtheorem{prop}[equation]{{\sc Proposition}}

\theoremstyle{definition}
\newtheorem{defn}[equation]{{\sc Definition}}

\theoremstyle{remark}
\newtheorem{rem}[equation]{Remark}
\theoremstyle{assumption}

\theoremstyle{question}

\date{August 2, 2024}

\title[$\epsilon$-logarithmic modes]{Entropy of logarithmic modes}

\author{Suresh Eswarathasan}
\address{Dalhousie University \\
Department of Mathematics $\&$ Statistics, Chase Building Rm 316 \\
Coburg Road \\
Halifax, Nova Scotia \\
Canada}
\email{sr766936@dal.ca}

\begin{document}

\begin{abstract}
Let $(M,g)$ be a compact, boundaryless, Riemannian manifold whose geodesic flow on its unit sphere bundle is Anosov.  Consider the (semiclassical) Laplace-Beltrami operator on $M$.  Let $\epsilon >0$.  

We study the semiclassical measures $\mu_{sc}$ of quasimodes spectrally supported in intervals of width $\epsilon \frac{\h}{|\log \h|}$, a critical-type regime when considering ``delocalization".  We derive a lower bound for the Kolmogorov-Sinai entropy of $\mu_{sc}$ that depends explicitly on $\epsilon$, in the spirit of that given by Ananthamaran-Koch-Nonnenmacher.
\end{abstract}

\maketitle

\tableofcontents

\section{Introduction} 

\subsection{Motivations}

The article contributes to the theory of semiclassical measures, important objects in quantum chaos, with an emphasis on entropy.  We give some review of these notions and pertinent results before stating our main theorem in Section \ref{sect:main_results}.

\subsubsection{Eigenfunctions and QUE}

An important problem in the field of quantum chaos is to understand the relationship between a dynamical system and its high-energy quantum mechanical models.  The Bohr correspondence principle says that the quantum mechanics replicate the classical mechanics in the semiclassical limit $\h \rightarrow 0.$  This principle can be observed in classical systems that exhibit complete integrability.  When the dynamics become more complicated, the usual techniques on the quantum side breakdown beyond ``short" times, that is, those techniques are valid up to the Ehrenfest time $C |\log \h|$ where $C>0$ is a constant determined by the classical system.  As a result of this break down and the necessity to understand the phenomenon of quantum evolution up to much longer times, a precise description of the stationary states (eigenfunctions) of the quantum system is difficult to obtain.

In this article, we will focus on the quantum system generated by the Laplace-Beltrami operator $\Delta_g$ on a compact Riemannian manifold $(M,g)$ without boundary.  To describe the high-frequency quantum dynamics, we will use a semiclassical formalism, that is we use the semiclassically rescaled Laplacian
\begin{equation*}
    P(\h) = -  \frac{\h^2 \Delta_g}{2}
\end{equation*}
and consider its spectrum in the vicinity of a fixed value $E = \frac{1}{2}$ in the the regime of $\h$ near 0.  The dynamics of a free quantum particle is captured by the Schr\"odinger equation
\begin{equation*}
    i \h \partial_t \phi(t) = P(\h) \phi(t), \, \, \, \phi(t) \in L^2(M, dV_g) 
\end{equation*}
where $dV_g$ is the Riemannian volume measure generated by $g$.  We call $\phi(t)$ the wavefunction of the particle at time $t.$  In the semiclassical limit, this dynamic can be related to the free motion of a classical particle on $M$, namely the geodesic flow.  The geodesic flow $g^t$ occurs in $T^*M$, whose points are denoted by $(x;\xi)$, and is actually the Hamiltonian flow associated to the (Hamiltonian) function $\frac{\|\xi\|_g^2}{2}$.  Thanks to being at energy $1/2$, and a homogeneity property, we consider can $g^t$ on the unit cotangent bundle $S^*M = \{\|\xi\|_g^2 = 1\}.$

Another assumption we make in this article is that $g^t$ exhibits the ``ideal" chaotic behavior: that is, it satisfies the Anosov property. The theory of dynamical systems shows such flows are ergodic and mixing with respect to the Liouville measure on $S^*M.$  Anosov flows have been studied extensively and encompass a large body of work.  To repeat what was expressed earlier: one objective of quantum chaos is to understand the localization properties of eigenfunctions of $P(\h)$ at energies close to $1/2$ in the presence of such classical dynamics. The semiclassical regime then amounts to studying the eigenfunctions $\phi_j$ of $-\Delta_g$ with eigenvalues close to $\h_j^{-2}$ (or rather, frequencies close to $\h_j^{-1}$), and letting these values tend to $\infty.$

The Quantum Ergodicity Theorem \cite{Sch74, CdV85, Zel87} shows that if the geodesic flow is ergodic (thus including those that are Anosov), then almost all the high-frequency eigenfunctions $\phi_j$ are asymptotically equidistributed over $M$.  A more precise consequence is the following: there exists a density-one subsequence $\{j_k\}_k \subseteq \mathbb{N}$ such that the probability measures $|\phi_{j_k}(x)|^2 dV_g(x)$ (after appropriately normalizing our eigenfunctions) weakly converge to the normalized $dV_g(x)$.  Actually, this phenomenon occurs in a more general sense (using Wigner measures) to be described shortly.  Versions of this result in settings with boundary, and concerning Dirichlet eigenfunctions, are due to G\'erard-Leichtnam \cite{GL93} and Zelditch-Zworski \cite{ZZ96}.

It is natural to ask about the localization properties of the remaining zero-density subsequences of eigenfunctions.  The \textit{Quantum Unique Ergodicity (QUE)} conjecture of Rudnick-Sarnak \cite{RS94} states that on manifolds $M$ with negative sectional curvatures, the entire sequence of eigenfunctions equidistributes.  At present, only a few cases of QUE have been successfully resolved: so-called Arithmetic QUE proven by Lindenstrauss \cite{Lin06} in the compact case, Soundararajan \cite{S10} in the non-compact case, and Silberman-Venkatesh for locally symmetric spaces \cite{SV07}.  These results take place on congruence surfaces that carry additional number-theoretic structure and admit operators which enjoy these symmetries, namely Hecke operators.  Specifically, Arithmetic QUE holds for joint eigenfunctions of the Laplace-Beltrami operator and the Hecke operators.  We highlight the role \textit{entropy} plays in the proofs of Lindenstrauss and Silberman-Venkatesh, a concept which is central to this article and will be discussed more shortly.  An analogue of Arithmetic QUE was established by Kurlberg-Rudnick \cite{KR20} for the model of the quantum cat map on the 2-torus.  See the work of Zelditch \cite{Zel92} for QUE in the context of random bases on spheres, as well as Kelmer \cite{Kel10} and Hassell \cite{Hass10} for counterexamples in other settings. 

The QUE conjecture, stated using semiclassical formalism, is believed to hold more generally for \textit{semiclassical measures}, which themselves are defined through Wigner distributions $W_{\h}$.  The $W_{\h}$ are quantities on the cotangent bundle $T^*M$ that simultaneously describe the localization of an $L^2(M,dV_g)$-normalized sequence $\{\Psi_{\h}\}_{\h}$ in position and momentum, and are an important object in this article.  The Wigner distribution associated to a $L^2(M, d V_g)$-normalized sequence of states $\{\Psi_{\h}\}_{\h}$ is defined as
\begin{equation*} 
W_{\h}(a) := \langle \Op_{\h}(a) \Psi_{\h}, \Psi_{\h} \rangle
\end{equation*}
where $a \in C^{\infty}_0(T^*M)$ and $\Op_{\h}(\bullet)$ is a quantization procedure that sends $a$ to a particular bounded operator on $L^2(M, d V_g)$; this procedure is described in further detail in Appendix \ref{sect:semiclassics}.  We call \textit{semiclassical measures} $\mu_{sc}$ the limit points of the sequence $\{W_{\h}\}_{\h}$ in the weak-* topology on the space of distributions.  It is well-known \cite{Z12} that semiclassical measures $\mu_{sc}$ associated to eigenfunctions of $P(\h)$ are $g^t$-invariant probabilty measures and are supported on, what we rename as, the energy shell 
\begin{align} \label{eqn:energy_shell}
\mathcal{E} := S^*M = \{ (x,\xi) \in T^*M \, | \, \|\xi\|_{g(x)} = 1 \}.
\end{align}
QUE in its fullest strength states the following: in the case of eigenfunctions on manifolds with negative sectional curvatures, $W_{\h}$ converges to the Liouville measure $dL$ on $\mathcal{E}$.  Note that the Quantum Ergodicity Theorem only guarantees this particular convergence of the Wigner distributions for a density-one subsequence. An important measure of localization, especially in the context of semiclassical measures, is the notion of \textit{Kolmogorov-Sinai entropy} and our main result (appearing in Section \ref{sect:main_results}) concerns this quantity. 

In closing this section, it is important to note that there are toy models in quantum chaos that demonstrate non-QUE behavior.  For example, quantum cat maps (without the arithmetic structure utilized in \cite{KR20}) satisfy quantum ergodicity \cite{BD96} yet do produce exceptional subsequences whose semiclassical measures are half-Lebesgue and half-atomic.  Such examples of subsequences were constructed by Faure-Nonnenmacher-de Bi\`evre \cite{FND03}.  It is widely suspected that large spectral multiplicities play a role in the phenomena demonstrated by these models, and it is this intuition that is our \textit{first motivation} for considering quasimodes.

\subsubsection{Quasimodes: Localization and Delocalization}

The Laplace-Beltrami operator on manifolds with Anosov geodesic flows is \textit{not expected} to have large spectral multiplicities, thus one might expect the equidistribution phenomena.  In response, one can artificially introduce ``large multiplicites" by considering functions that act as approximations to true eigenfunctions and are called \textit{quasimodes}. The order of approximation can be quantified with the notion of \textit{width} and is given in the following definition:

\begin{defn}[Quasimodes]\label{defn:qmodes}
For a given energy level $E>0$, we say that a semiclassical family $\{\Psi_{\h}\}_{\h}$ of $L^2$-normalized states is a family of \emph{quasimodes of $P(\h)=-\h^2 \frac{\Delta_g}{2}$ of central energy $E$} and \emph{width} $w(\h)$, if and only if there exists $\h_0(E)>0$ such that 
\begin{equation}\label{eqn:QM1}
\| (P(\h)- E) \Psi_{\h} \|_{L^2(M)}  \leq  w(\h)\,,\qquad \forall \h \leq \h_0 \,.
\end{equation}
\end{defn}

Normally $w(\h) = \mathcal{O}(1)$ in $\h$.  In this paper, we take $E=\frac{1}{2}$ and $w(\h) \rightarrow 0$ at rate to be specified in the next section.  First, note that linear combinations of eigenfunctions corresponding to eigenvalues in a $w(\h)$-width window form a class of examples. These act as stand-ins when our spectrum exhibits low multiplicities. Second, also note that quasimodes of width $o(\h)$ have semiclassical measures $\mu_{sc}$ also being $g^t$-invariant, probability, and supported on $\mathcal{E}$.

Our \textit{second motivation} for considering quasimodes comes from the difficulty of finding explicit expressions for eigenfunctions in the presence of chaotic classical dynamics.  Except for special systems like the quantum cat maps, we are unable to give even approximate expressions of eigenfunctions in Anosov systems.  However, one can actually construct explicit quasimodes of certain widths on compact manifolds with Anosov geodesic flows.

A trio of results due to Brooks \cite{Br13}, Eswarathasan-Nonnenmacher \cite{EN17}, and Eswarathasan-Silberman \cite{ES18} show how spectral degeneracy (in the proposed form of \`a la quasimodes) can lead to phenomena other than equidistribution, as inspired by the work of Faure-Nonnenmacher-de Bi\`evre \cite{FND03} on the quantum cat map.  In particular on a compact surface having a closed hyperbolic geodesic, generalizing Brooks' result, \cite{EN17} constructs a sequence of quasimodes of width $C \frac{\h}{|\log \h|} $ with central energy $E$ and whose semiclassical measure is a delta function on that geodesic; in fact, this result continues to hold for more general quantum Hamiltonians whose associated classical dynamics has at least one closed hyperbolic orbit.  The works \cite{Br13, ES18} are particular to constant negative curvature with that of Eswarathasan-Silberman holding in higher dimensions but Brooks' ideas proving to be robust enough for applications to $L^p$ norms \cite{HuSog24}.

In fact, the results of Brooks, Eswarathasan-Nonnenmacher, and Eswarathasan-Silberman were motivated in part by the works of Anantharaman \cite{An08}, Anantharaman-Koch-Nonnenmacher \cite{AN07, AKN06}, Rivi\`ere \cite{Riv10}, Anantharaman-Silberman \cite{AS12}, and Brooks-Lindenstrauss \cite{BL14}.  In the second listed of these inspiring works, it was specified that semiclassical measures associated to quasimodes of width $o \left( \frac{\h}{|\log \h|} \right)$ must satisfy the same ``half-delocalization" constraints as those associated with eigenfunctions through the use of entropy. The works of Dyatlov-Jin-Nonnenmacher \cite{DJ18, DJN19}, underpinned by an earlier breakthrough of Bourgain-Dyatlov \cite{BD18} for hyperbolic surfaces, show the full support of any $\mu_{sc}$ holds for quasimodes of width $o \left( \frac{\h}{|\log \h|} \right)$.  Thus, for such Anosov flows, the logarithmic scale $c \frac{\h}{|\log \h|}$ is a sort of critical value when considering ``delocalization" in the context of these results.  This puts us in a position to introduce our article's main objects of study:

\begin{defn}[$\epsilon$-logarithmic modes] \label{defn:log_modes}
For a given energy level $E>0$ and width parameter $\epsilon>0$, we say that a semiclassical family $\{\Psi_{\h}\}_{\h}$ of $L^2$-normalized states is a family of \textit{$\epsilon$-logarithmic modes} for $P(\h)=-\h^2 \frac{\Delta_g}{2}$ if for the intervals
\begin{equation*}
I(\epsilon,\h) := \left[ E - \epsilon \, \frac{\h}{|\log \h|} , E + \epsilon \, \frac{\h}{|\log \h|} \right]
\end{equation*}
and an $\h_0(E, \epsilon)>0$ such that for all $\h \leq \h_0$, the modes $\Psi_{\h}$ are spectrally supported in $I(\epsilon, \h)$: that is, $\Pi_{I(\h)^{\complement}} \Psi_{\h} = 0$.  In other words, $\{\Psi_{\h}\}_{\h}$ is a $L^2$-normalized family consisting of linear combinations of eigenfunctions whose eigenvalues sit inside of $I(\epsilon,\h)$ therefore making it a family of quasimodes for $P(
h)$ with energy $E$ and width $w(\h) = \epsilon \frac{\h}{|\log \h|}$.
\end{defn}

\subsection{Main result} \label{sect:main_results}

Recall that $(M,g)$ is a $d$-dimensional, compact, and boundaryless manifold whose geodesic flow $g^t$ on $\mathcal{E}$ is Anosov.  To better understand the various semiclassical measures $\mu_{sc}$ of $\epsilon$-logarithmic modes, we aim to compute their \textit{Kolmogorov-Sinai entropies}.  The Kolmogorov-Sinai entropy of a $g^t$-invariant probability measure $\mu$ is a non-negative number $H_{KS}(\mu)$ describing the long-time, average information carried by $\mu$-typical geodesics.  A more precise definition of  $H_{KS}$ is provided in Section \ref{sect:entropy_quasiproj} of the appendix with further background found in \cite[Chapter 9]{BS}.  

A classical upper bound on $H_{KS}$ for a $g^t$-invariant probability measure $\mu$ is to due to Ruelle \cite{Rue78}:
\begin{equation*}
H_{KS}(\mu) \leq  \int_{\mathcal{E}} \log J^u(\rho) \, d \mu(\rho).
\end{equation*}
Here, 
\begin{equation*}
J^u(\rho) : = |\det D(g^1)_{\upharpoonright E^u(\rho) }|
\end{equation*}
is the \textit{unstable Jacobian for $g^1$ at $\rho \in \mathcal{E}$} with $E^u(\rho)$ being the unstable subspace at $\rho \in \mathcal{E}$; recall that in our geometry, $\mathcal{E}$ admits a foliation into unstable manifolds of the flow \cite[Chapter 5]{BS}. Although $\log J^u(\rho)$ depends on the metric structure on $\mathcal{E}$, its average over any invariant measure does not and is always positive.  Note that the unstable Jacobian takes a slightly different definition in the previous works \cite{An08, AN07, AKN06, Riv10} and that we use the formulation given in \cite[Section 1.2]{Non12}.  

In the case of Anosov systems \cite{LY85}, equality in Ruelle's upper bound occurs if and only if $\mu = dL$.  When the sectional curvatures $\mathcal{K}(M) = -1$, the measure of maximal entropy is $dL$ and thus $H_{KS}(\mu) \leq d-1$ for any $g^t$-invariant probability measure $\mu$. Thus, in this case, the QUE conjecture can be reformulated as claiming $H_{KS}(\mu_{sc}) = d-1$.

Now, define
\begin{equation} \label{eqn:max_exp_rate}
\lambda_{\max} := \lim_{t \rightarrow  \infty} \sup_{\rho \in \mathcal{E}} t^{-1} \log \|D(g^t)\|
\end{equation}
to be the \textit{maximal expansion rate on $\mathcal{E}$}.  Next, we let 
\begin{equation} \label{eqn:LambdaMax}
\Lambda := \sup_{\rho \in \mathcal{E}} \log J^u(\rho), 
\end{equation}
a quantity related to the maximal volume expanding rate of the unstable manifold; see also \cite[Theorem 1.1.1.]{An08}.  And lastly, let $H_{top}$ be the topological entropy of $g^1_{\upharpoonright \mathcal{E}}.$  Our main result is:
\begin{theo} \label{thm:main}
Let $(M,g)$ be a $d$-dimensional compact, boundaryless, Riemannian manifold whose geodesic flow on $\mathcal{E}$ is Anosov.  For all $\epsilon <  1$, and any $\epsilon$-logarithmic family $\{ \Psi_{\h} \}_{\h}$ with energy $\frac{1}{2}$ and semiclassical measure $\mu_{sc}$, 
\begin{align*}
   \nonumber  H_{KS}(\mu_{sc})  \geq  \int_{\mathcal{E}} \log \left( J^u(\rho) \right) \, & d\mu_{sc}(\rho) - \frac{(d-1)\lambda_{\max}}{2} \\ 
   &  -  \frac{3}{2}  \, H_{top}  \, \frac{\epsilon}{\lambda_{\max}} - \Lambda\, \left( \frac{2 \epsilon}{\lambda_{\max}} +  \left( \ \frac{\epsilon}{\lambda_{\max}}\right)^{2} \right).
\end{align*}
\end{theo} 

\begin{cor}
    Under the hypotheses of Theorem \ref{thm:main} with $(M,g)$ being a hyperbolic surface (that is $d=2$ with $\mathcal{K}(M)=-1$), we have positive entropy for $\epsilon < \epsilon_0 \approx .137 .$
\end{cor}

\begin{rem} Our proof can be used to recover the main results of \cite{AN07,AKN06} when $\epsilon=o(1)$ in $\h$.  However, the novelty of this result is the following. Recall the results of Brooks, Eswarathasan-Nonnenmacher, and Eswarathasan-Silberman \cite{Br13, EN17, ES18} discussed in the previous section.  Our main result demonstrates that such strong localization (leading to zero entropy) cannot occur for our logarithmic modes, particularly when $\epsilon$ is sufficiently small depending on the given dynamical quantities. 
\end{rem}

\begin{rem}
While Theorem \ref{thm:main} takes place in the context of the semiclassical Laplace-Beltrami operator $-\h^2 \frac{\Delta_g}{2}$, the result should generalize with relative ease to other quantum Hamiltonians $P(\h)$ satisfying some basic non-degeneracy and spectral conditions near a fixed energy $E>0$ and whose symbol generates an Anosov Hamiltonian flow.  Furthermore, our result should also extend to $\epsilon$-logarithmic modes but associated to central energies $E_{\h} \sim E$.  In fact, such an extension would show that the partially localizing modes associated to energies $E_{\h}$ also constructed in \cite{Br13, EN17, ES18} exhibit positive entropy. For clarity of exposition, we do not pursue this here due to the current level of technicality. 
\end{rem}

\subsection{Outline of the proof} 

We review the ideas of Anantharaman and Nonnenmacher and then describe our new contributions, ignoring some technical parameters for now.  All of the upcoming objects are rigorously defined in Sections \ref{sect:quantum_measures}-\ref{sect:quant_entropies_EUP} and Appendix B.

\subsubsection{Review of \cite{AN07, AKN06}} First, let $n_0 \in \mathbb{N}$ and consider $q \in \mathbb{N}$ such that $qn_0 \leq \lfloor \frac{|\log \h|}{\lambda_{\max}} \rfloor$, the right-hand being twice the ``Ehrenfest time".  Next, for $K \in \mathbb{N}$, set $\Sigma(qn_0) = \{1,\dots,K\}^{qn_0}$ and consider this as the set of words of length $qn_0$ on $K$ letters.  These will correspond to particular dynamical quantities, specifically $g^{t}$-refined partition elements at time $qn_0$.

One can attach a bounded operator $\Pi_{\bm{\beta}}$ to such a $qn_0$-length word $\bm{\beta}$ and which serves as a ``refined" element appearing in a quantum partition of unity.  Moreover, the quantity $\| \Pi_{\bm{\beta}} \Psi_{\h} \|^2$ is the quantum analogue of a $g^t$-invariant probability measure evaluated at the word $\bm{\beta}$.  These notions allow us to analyse a type of quantum entropy at time $qn_0$, namely $H_0^{qn_0-1}(\Psi_{\h})$.  Finally, let $J_{\bm{\beta}}^u$ be a discretized version of the unstable Jacobian $J^u(\rho)$ associated to the cylinder $\bm{\beta}$.  

The work of Anantharaman-(Koch)-Nonnenmacher \cite{AKN06, AN07} establishes that within the described regime of $n_0$ and $\h$ (modulo errors in a variety of parameters including $\h$ and using some invariance statements),
\begin{align*}
H_0^{qn_0-1}(\Psi_{\h}) - \sum_{\bm{\beta} \in \Sigma(qn_0)} \log(J_{\bm{\beta}}^u) \| \Pi_{\bm{\beta}} \Psi_{\h} \|^2 \geq -\frac{(d-1)\lambda_{\max}}{2} qn_0
\end{align*}
thanks to a combination of the entropic uncertainty principle (Section \ref{sect:EUP}), which involves a time symmetrization, and the hyperbolic dispersive estimate (Appendix, Section C \ref{sect:hyp_disp}).  In order to make rigorous the passage from quantum entropies at long times $H_0^{qn_0-1}(\Psi_{\h})$ to actual metric entropy for $\mu_{sc}$ at finite times $n_0$, one uses a sub-additivity argument to reduce the above inequality (through the ``entropic uncertainty principle" that involves a time symmetrization) to 
\begin{align*}
 & H_0^{n_0-1}(\Psi_{\h}) + \sum_{j=1}^{q-1} H_0^{n_0-1}(e^{-i \frac{ jn_0P(\h)}{\h}}\Psi_{\h})   - q\sum_{\bm{\beta} \in \Sigma(n_0)} \log(J_{\bm{\beta}}^u) \| \Pi_{\bm{\beta}} \Psi_{\h} \|^2 \\
 & - \sum_{\bm{\beta} \in \Sigma(n_0)} \log(J_{\bm{\beta}}^u) \sum_{j=1}^{q-1} F_1(\h,  \epsilon, \bm{\beta}, jn_0) \geq -\frac{(d-1)\lambda_{\max}}{2} qn_0
\end{align*}
where the operators $\Pi_{\bm{\beta}}$ are pseudodifferential with now $\h$-independent principal symbols, thanks to the word lengths being independent of $\h$, and $F_1(\h, \epsilon, \bm{\beta}, jn_0) =   \| \Pi_{\bm{\beta}} e^{-i \frac{ jn_0P(\h)}{\h}} \Psi_{\h} \|^2 - \| \Pi_{\bm{\beta}} \Psi_{\h} \|^2 $.  At this stage one can make further use of semiclassical methods.  One goes from these quantum entropies to metric entropies after dividing both sides by $qn_0$ and carefully taking the limits $\h \rightarrow 0$, $n_0 \rightarrow \infty$, and $K \rightarrow \infty$ (amongst others limits of parameters we previously stated we would ignore).

\subsubsection{New contributions}
When $\Psi_{\h}$ is an eigenfunction, we have 
\begin{align*}
\sum_{j=1}^{q-1} H_0^{n_0-1}(e^{-i \frac{ jn_0P(\h)}{\h}}\Psi_{\h}) = (q-1) H^{n_0-1}_0(\Psi_{\h}) \mbox{ and } \sum_{\bm{\beta} \in \Sigma(n_0)} \log(J_{\bm{\beta}}^u) \sum_{j=1}^{q-1} F_1(\h, \epsilon, \bm{\beta}, jn_0) = 0
\end{align*}
modulo small errors (occuring in the equalities) in various parameters, including $\h$.  \textit{In the case of $\epsilon$-logarithmic modes, these quantities are more complicated with part of our article showing how to control them in $\epsilon$ and other variables.}  The generalized sub-additivity statements incorporating the additional spectral mass are given in Section \ref{sect:subadd}.  The majority of Section \ref{sect:defect} is dedicated to estimating the defect
\begin{align*}
D(h,\epsilon, n_0) := q^{-1} \sum_{j=0}^{q-1} \left( H_0^{n_0-1}(e^{-i \frac{ jn_0P(\h)}{\h}}\Psi_{\h}) - H_0^{n_0-1}(\Psi_{\h}) \right)
\end{align*}
by using a total-variation-type metric to quantify the distance between the probability measures generated by $\Psi_{\h}$ and $e^{-i \frac{ jn_0P(\h)}{\h}} \Psi_{\h}$.  We spend the remainder of that section estimating our new ``potential" terms $\sum_{\bm{\beta} \in \Sigma(n_0)} \log(J_{\bm{\beta}}^u) \sum_{j=1}^{q-1} F_1(\bullet) $.

We end this section by emphasizing that the main reason for considering the types of spectral localization seen in Definition \ref{defn:log_modes} is simply to guarantee  $WF_{\h}(\Psi_{\h}) \subseteq \mathcal{E}$ and then easily employ the so-called entropic uncertainty principle (see Section \ref{sect:quant_entropies_EUP}).  Additionally, $\epsilon$-logarithmic modes are slightly more convenient when using spectral calculus at various places in our argument.  This being said, our results should generalize to quasimodes of width $\epsilon \frac{\h}{|\log \h|}$ with more work.

\bigskip

\noindent \textbf{Acknowledgements:}  SE would like to thank Lior Silberman, Gabriel Rivi\`ere, Filipo Morabito, Ze\'ev Rudnick, and Anthony Quas for various discussions leading to the final version of this article.  SE warmly thanks St\'ephane Nonnenmacher for the multiple ways he has helped: providing an informative counterexample for quantum cat maps that motivated SE to relinquish an assumption in the first version of this article, a crucial discussion (amongst others) which led to the writing of Section \ref{sect:defect}, and being supportive of the completion of the final version of this article.  SE was funded by the NSERC Discovery Grant program.

\section{List of various parameters} \label{sect:param_list}

We record a list of important parameters that appear frequently throughout our analysis.  Note that many have been given subscripts in word form for easier associations.  References to the appendices are given for parameters already appearing in \cite{AN07, Non12}.

\begin{itemize}
\item $\epsilon>0$ governs the width of our quasimode in Definition \ref{defn:log_modes}. 
\item $K$ is the number of elements in our initial partition $\mathcal{P}$.  In \cite{AN07}, $K \simeq \zeta^{-d}$ for $\zeta$ a small and positive parameter that allows one to replace finite sums over the initial partition $\mathcal{P}$ with integrals.  See Definition \ref{defn:classical_entropy}.  
\item $n_0 \in \mathbb{N}$ is our fixed and finite time, independent of $\h$, for the dynamics of $\mu_{sc}$.  This is eventually taken arbitrarily large.
\item $\h_0> 0 $, such that $\h \leq \h_0$, is determined by a number of restrictions including the quasimode estimate.  This will depend on most of the parameters listed below.
\item $\kappa_{slab}>0$ governs the width of our neighbourhood around $\mathcal{E}$, namely $\mathcal{E}_{\kappa_{slab}}$.  See Definition \ref{def:smooth_partition}.  We eventually set $\kappa_{slab}=K^{-1}.$
\item $\delta_{energy} > 0$ is a free parameter first appearing in the width of our energy cutoff $\chi^{(n)}(P(\h) - \onehalf)$.  This cutoff appears in the hyperbolic dispersive estimate, the parameter itself appearing in the loss incurred in $\h$.   See equation (\ref{eqn:small_scale_cutoff}) and Proposition \ref{prop:hyp_disp}.  The parameter will play a role in Section \ref{sect:proj_micro_meas}. 
\item $\delta_{Ehr} \in (0,1)$ is a small parameter.  It appears in the expression for the Ehrenfest time: $T_{\delta_{Ehr},\kappa_{slab}, \h} = \frac{(1-\delta_{Ehr}) |\log \h|}{2 \lambda_{\max}(\kappa_{slab})}$.  Note this has an effect on the regularity of our symbols.  In Sect \ref{sect:quantum_measures}), we set $\delta_{Ehr}=\delta_{energy} < \frac{2 \lambda_{max}}{100(1+2 \lambda_{max})}$ in order to have $1 > \delta_{energy}(1 + \frac{1 - \delta_{Ehr}}{2})$, itself satisfying the hypotheses of the hyperbolic dispersive estimate.  See Definition \ref{defn:Ehrenfest}.
\end{itemize}

\section{Microlocalization} \label{sect:proj_micro_meas}

Recall that on compact $(M,g)$, the spectrum of $P(\h) = -\h^2 \frac{\Delta_g}{2}$ is discrete and infinite, which we denote by $\{E_j(\h)\}_{j \in \N}$.  Our space $L^2(M, dV_g)$ admits an orthonormal basis of eigenfunctions, whose individual elements we write as $\psi_j$.  In particular, if $\Psi_{\h}$ is $L^2$-normalized, then 
\begin{equation*}
\Psi_{\h} = \sum_j c_j(\h) \, \psi_j
\end{equation*}
where $\sum_j c_j(\h)^2 = 1$. 

\subsection{Microlocalization properties} \label{sect:micro_props}
The main statement of this section is a lemma on the wavefront set for our $\epsilon$-logarithmic modes.  We reproduce the definition of $\chi^{(n)}$, an important cutoff to be quantized and use repeatedly throughout our note, from Appendix \ref{sect:aniso_sharp_cutoffs}.

Let  $\delta_{energy}>0$ be given and $n \leq C_{\delta_{energy}} |\log \h|$ where $C_{\delta_{energy}} + 1 < \delta^{-1}_{energy}$.  Define $\chi_{\delta_{energy}}(s)$ equal to 1 for $|s| \leq e^{-\delta_{energy}/2}$ and 0 for $|s| \geq 1$.  From here we set
\begin{equation} \label{eqn:small_scale_cutoff}
\chi^{(n)}(s,\h) := \chi_{\delta_{energy}}(e^{-n\delta_{energy}} h^{-1+\delta_{energy}} s).
\end{equation}

\begin{lem}[$\epsilon$-logarithmic modes are ``microlocalized near $\mathcal{E}$ at scale $\h^{\gamma}$"] \label{lem:strong_micro}
Consider $\chi^{(n)}$ as in (\ref{eqn:small_scale_cutoff}) and $\delta_{energy}>0$ be given.  Let $\{\Psi_{\h}\}_{\h}$ be a family of $\epsilon$-logarithmic modes and $C_{\delta_{energy}} < (\delta_{energy})^{-1} -1 $.  Then there exists $\h_0(\epsilon, \delta_{energy})$ such that for all $\h \leq \h_0$ and $n \leq C_{\delta_{energy}} |\log \h|$, we have
\begin{equation*}
\left( Id - \chi^{(n)} \left( P(\h)- \frac{1}{2} \right) \right) \Psi_{\h} = 0
\end{equation*}
where $\chi^{(n)}(P(\h)- \onehalf)$ is defined through spectral calculus. Furthermore, $WF_{\h}(\Psi_{\h}) \subseteq \mathcal{E}$.
\end{lem}

\begin{rem}
  If one uses the second-microlocal calculus defined in \cite{SjZw99} (see also Section \ref{sect:aniso_sharp_cutoffs}) and whose quantization procedure is denoted by $\Op_{\h,\mathcal{E}} \left( \bullet \right)$, then $\chi^{(n)}(P(\h)- \onehalf) - \Op_{\mathcal{E}, \h} \left(  \chi^{(n)}( \|\xi \|^2 - \frac{1}{2} ) \right) = \mathcal{O}(\h^{\infty})$ as $\h \rightarrow 0.$
\end{rem}

\begin{proof}[Proof of Lemma \ref{lem:strong_micro}]
This is a simple computation via spectral calculus after using the support properties of $\chi^{(n)}$.  Thanks to $\chi^{(n)}\left( P(\h) - \frac{1}{2} \right)$ being a function of $P(\h)$, we have 
\begin{equation*}
\chi^{(n)} \left( P(\h) - \frac{1}{2} \right) \Psi_{\h} = \sum_{E_j \in I(\epsilon, h)} \chi^{(n)} \left( P(\h) - \frac{1}{2} \right) c_j \psi_j
\end{equation*}
where $\psi_j$ is an eigenfunction with eigenvalue $E_j$.  Thus, $\sum_{E_j \in I^{\complement}(h)} \left( 1- \chi^{(n)}  \left( E_j - \frac{1}{2} \right) \right) c_j \psi_j = 0 $ as $c_j = 0$ thanks to $\chi^{(n)}$ having width less than a fixed multiple of $h^{1-(1+C_{\delta_{energy}})\delta_{energy}}$ and $I$ having width $\frac{\h}{|\log \h|}$, where $(1+C_{\delta_{energy}})\delta_{energy} < 1$.

The wavefront set statement follows from the use of the second-microlocal calculus in the form of Lemma \ref{lem:disj_supp}: given any $a \in S^0$ such that $\supp a \cap \mathcal{E} = \emptyset$, $\Op_{\mathcal{E}, \h}(a) \Psi_{\h} = \Op_{\mathcal{E}, \h}(a) \chi^{(n)} \left( P(\h) - \frac{1}{2} \right) \Psi_{\h} + \mathcal{O}(\h^{\infty}).$ But as $\h \rightarrow 0$, $\Op_{\mathcal{E}, \h}(a) \chi^{(n)} \left( P(\h) - \frac{1}{2} \right) = \mathcal{O}(\h^{\infty})$ thanks $\mathcal{E}$ being compact.
\end{proof}

\begin{rem}
Note that for more general quasimodes of width say $\mathcal{O} \left( \frac{\h}{|\log \h|} \right)$, it is not clear that the wavefront set is a subset of $\mathcal{E}$, which allows for an $\mathcal{O}(h^{N})$ $L^2$ estimate microlocally away from $\mathcal{E}$ for some $N$ sufficiently large. In fact, one gets the same $\mathcal{O} \left( \frac{\h}{|\log \h|} \right)$ estimate on $\Psi_{\h}$ away from the support of $\chi^{(n)}(P(\h) - E)$ if using standard parametrix arguments.  This will be important in applying the entropic uncertainty principle found in Proposition \ref{prop:EUP}. 
\end{rem}

\section{Quantum measures} \label{sect:quantum_measures}

This section heavily uses the material, particularly the notation, set up in Appendix \ref{sect:entropy_quasiproj}.  If the reader is unfamiliar with the previous works \cite{AN07, AKN06, Non12}, it is suggested they familiarize themselves with various dynamical and operator-theoretic quantities (and the passages between them) including the initial dynamical partition $\mathcal{P}$,  the words $\bm{\alpha} \in \Sigma^n$, the operators $\Pi_{\bm{\alpha}}$, and the smoothed quantum partition $\mathcal{P}_{\rm{sm,q}}$.  These can be found in Appendix \ref{sect:entropy_quasiproj}, particularly Appendices \ref{sect:entropy_smooth} and \ref{sect:quasiprojs}.  Note that Proposition \ref{prop:pou_quasi} will play an important role in this section.  

We specify our choice of $\mathcal{P}$, described also in Appendix \ref{sect:dynam_defs}: we first take an open cover $\{ \mathscr{O}_k'\}_{k'}$ of $M$ with $M = \sqcup_{k=1}^K E'_k$, a partition consisting of Borel sets $E_k' \subseteq \mathscr{O}_{k}'$.  Next, we take the elements $E_k$ as $\pi^{-1}(E'_k)$ where $\pi: \mathcal{E} \rightarrow M$ is the canonical projection.  Note that there exists $\zeta_g > 0$ such that if $diam \{ \mathscr{O}_k' \}_k \leq \zeta_g$, then $\mathcal{P}^{\vee 2}$ is generating \cite[Page 7]{AKN06}, so we assume that $K = \lfloor \zeta_g^{-1} \rfloor.$  The diameter of $\mathcal{P}$ is allowed to go to 0 in \cite{AN07, Non12, AKN06}, with this diameter playing a non-trivial role in our work and thus requiring us to pay close attention in Section \ref{sect:final_steps}.

We recall an important time scale in semiclassical analysis.  Let $\delta_{Ehr}, \kappa_{slab}>0,$ and $\h>0$ be given; from here onward we set $\delta_{energy}=\delta_{Ehr} < \frac{2 \lambda_{max}}{100(1+2 \lambda_{max})}$ as in the hypothesis of Lemma \ref{lem:strong_micro} and (\ref{eqn:max_exp_rate}).  We call
\begin{equation*} 
T_{\delta_{Ehr}, \kappa_{slab}, \h} := \frac{(1-\delta_{Ehr})}{2 \lambda_{\max}(\kappa_{slab})} |\log \h|
\end{equation*}
the \textit{Ehrenfest time} on the energy slab $\mathcal{E}_{\kappa_{slab}} = \{ (x,\xi) \in T^*M \, | \, \frac{1}{2} - \kappa_{slab} \leq \| \xi \|_{g(x)} \leq \frac{1}{2} + \kappa_{slab} \}.$  The number $\lambda_{\max}(\kappa_{slab})$ is the maximal expansion rate across $\mathcal{E}_{\kappa_{slab}}$.

\subsection{Quantum measures and words}

We have a natural notion for a probability measure on the space of words $\Sigma^n$, as follows:

\begin{defn}[Quantum measures]
Let $\{\Psi_{\h}\}_{\h}$ be an $\epsilon$-logarithmic family and take $C_{\delta_{Ehr}}>0$.  Then for any $\bm{\alpha} \in \Sigma^n$ with $n \leq C_{\delta_{Ehr}} |\log \h|$, we write
\begin{equation*}
\mu_{\h}(\bm{\alpha}) := \| \Pi_{\bm{\alpha}} \Psi_{\h} \|^2
\end{equation*}
and call it the \textit{forward (quantum) measure}.  For $\alpha \in \Sigma^0$, define $\Pi_{\alpha}(-n) := U^{n} \Pi_{\alpha} U^{-n}$.  Analogous to the forward meausre, we define the \textit{backwards (quantum) measure} for $\{\Psi_{\h}\}_{\h}$ as
\begin{equation*}
\tilde{\mu}_{\h}(\bm{\alpha}) := \| \Pi_{\alpha_{-n}}(-n) \dots \Pi_{\alpha_{-1}}(-1)\Psi_{\h} \|^2,
\end{equation*}
\end{defn}
More formally, $\mu_{\h}$ and $\tilde{\mu}_{\h}$ can be seen as equivalence classes of $\h$-dependent probability measures on words of length $n \leq C_{\delta_{Ehr}} |\log \h|$, with the equivalence being up to $\mathcal{O}(\h^{\infty})$ quantities as $h \rightarrow 0$.  In other words, the value $\mu_{\h}(\bm{\alpha})$ is the weight of the corresponding cylinder in $\{1, \dots, K \}^{\mathbb{Z}_{\geq 0}}$.  For $n$ fixed, $\mu_{\h}(\bm{\alpha}) \rightarrow \mu_{sc}(\pi_{\bm{\alpha}})$ as $\h \rightarrow 0$ thanks to Egorov's Theorem for short times (that is, $\h$-independent times) with a similar formula for $\tilde{\mu}_{\h}$.  We will toggle between the notations of $\mu_{\h}(\bullet)$ and $\| \bullet \Psi_{\h} \|^2$ depending on the calculation.

We begin with a proposition, taken from \cite{AN07} but with their $\delta'=\delta_{Ehr}$ and their $\delta = \delta_{energy}$, that will be used repeatedly throughout the article:

\begin{prop} \label{p:strong_microlocal_quasi} \cite[Corollary 5.6]{AN07}
Let $\chi^{(n)}$ be the cutoffs considered in Appendix \ref{sect:aniso_sharp_cutoffs} and consider $C_{\delta_{Ehr}}>0$ is as in Lemma \ref{lem:strong_micro}.  Let $N >0$ be given.  There exists $\h_0(\delta_{Ehr}, n_0, N, K)$ such that for $\h \leq \h_0$, all $n \leq C_{\delta_{Ehr}} |\log \h|$, and $|\bm{\alpha}|=n$,
\begin{equation*}
    \| \left(I - \chi^{(n)} \left( P(\h) - \frac{1}{2} \right) \right) \Pi_{\bm{\alpha}} \chi^{(0)} \left( P(\h) - \frac{1}{2} \right) \| = \mathcal{O}(\h^{N}).
\end{equation*}
\end{prop}

Next, we give \cite[Equation (3.13)]{Non12} but encompassed in the form of the following: 

\begin{lem}[Compatibility conditions for logarithmic modes] \label{lem:compatible_meas}
Let $\{\Psi_{\h}\}_{\h}$ be an $\epsilon$-logarithmic family and $N>0$ be given.  Let $\delta_{Ehr} \in (0,1)$ and  $\kappa_{slab} >0$ be given.   Then there exists $\h_0(\epsilon, \delta_{Ehr}, \kappa_{slab}, K,N)$ such that for $\h \leq h_0$ and for all $ n \leq \lfloor 2 T_{\delta_{Ehr}, \kappa_{slab}, \h} \rfloor$ we have
\begin{align*}
& \sum_{\bm{\alpha} \in \Sigma^n} \| \Pi_{\bm{\alpha}} \Psi_{\h} \|^2 = 1 + \mathcal{O}_{ K}(\h^{N})
\end{align*}
and therefore $\forall \alpha_0, \dots, \alpha_{n-2} \in \Sigma^1$,
\begin{align*} 
\| \Pi_{\alpha_0,\dots, \alpha_{n-2}} \Psi_{\h} \|^2 = \sum_{\alpha_{n-1}}  \| \Pi_{\alpha_0,\dots, \alpha_{n-1}} \Psi_{\h} \|^2  + \mathcal{O}_{ K}(h^{N}).
\end{align*}
\end{lem}
\begin{proof}
This is an immediate implication of our previous results.  Thanks to Lemma \ref{lem:strong_micro}, for any $\chi \in S^{-\infty, 0}(T^*M)$ supported on $\mathcal{E}_{\epsilon_{slab }/2}$, $\Op_{\h}(\chi) \Psi_{\h} = \Psi_{\h} + \mathcal{O}(\h^{\infty})$.  Proposition \ref{p:strong_microlocal_quasi} shows that $\Pi_{\bm{\alpha}} \Psi_{\h}$ continues to be microlocalized near $\mathcal{E}$ at scale $\hbar^{\gamma}$ for some $\gamma>0$.  Now apply Proposition \ref{prop:pou_quasi} with the same choice of $\chi$.
\end{proof}

\subsection{Shift invariance}

For $\bm{\beta} \in \Sigma^{n_0}$, define $\Pi_{\bm{\beta}}(n) := U^{-n} \Pi_{\bm{\beta}} U^n$ and $\Pi_{\bm{\beta}}(-n) := U^{n} \Pi_{\bm{\beta}} U^{-n}$ which represent a ``forward shift" and ``backwards shift" of the quantum partition element.  Similarly to \cite[Proposition 4.1]{AN07} and \cite[Lemma 3.14]{Non12}, we have a more general statement for the invariance under the shift-type operation:

\begin{prop}[Approximate shift-invariance for logarithmic modes] \label{l:approx_shift} 
Consider an $\epsilon$-logarithmic family $\{\Psi_{\h}\}_{\h}$.  Fix some $n_0 \geq 0$, $\delta_{Ehr}>0$, and $\kappa_{slab}>0$.   Then, there exists $\h_0(n_0, \epsilon, \delta_{Ehr}, \kappa_{slab}, K)$ such that for all $\h \leq \h_0$, any $\bm{\beta} \in \Sigma^{n_0}$, and  $n$ in the range $[0, 2 T_{\delta_{Ehr}, \kappa_{slab}, \h} - n_0]$ we have
\begin{align} \label{e:shift_formula}
\nonumber & \sum_{\bm{\alpha} \in \Sigma^n} \left\|  \Pi_{\bm{\beta}}(n) \left[ \Pi_{\alpha_{n-1}}(n-1) \dots \Pi_{\alpha_{0}}(0) \right] \Psi_{\h} \right\|^2 = \langle |\Pi_{\bm{\beta}}|^2 U^n \Psi_{\h}, U^n \Psi_{\h} \rangle + G(\h, \bm{\beta},\epsilon, n) \\
& = \mu_{\h} \left(  \bm{\beta}  \right) +  F_1(\h, \bm{\beta}, \epsilon, n) + G(\h, \bm{\beta}, \epsilon, n),
\end{align}
where 
\begin{align} \label{eqn:F_1}
\nonumber  F_1(\h, \bm{\beta}, \epsilon, n) & := \langle |\Pi_{\bm{\beta}}|^2 U^n \Psi_{\h}, U^n \Psi_{\h} \rangle - \mu_{\h} \left( \bm{\beta}  \right) \\
& = \langle |\Pi_{\bm{\beta}}|^2 \Psi_{\h}, \Phi_{\h}[n]  \rangle  + \langle |\Pi_{\bm{\beta}}|^2 \Phi_{\h}[n], \Psi_{\h}  \rangle + \langle |\Pi_{\bm{\beta}}|^2 \Phi_{\h}[n], \Phi_{\h}[n]  \rangle,
\end{align}
\begin{equation*}
\Phi_{\h}[n](x) := \frac{-i}{\h} \int_0^n \left(P(\h) - \frac{1}{2} \right) U^{s} e^{is/2\h}\Psi_{\h} \, ds,
\end{equation*}
and 
\begin{align*} 
 G(\h, \bm{\beta}, \epsilon, n) 
 = \mathcal{O}_{K}(n \cdot \h^{\delta_{Ehr}}) . 
\end{align*}
Moreover, for $N>0$ given,
\begin{align*}
\sum_{\bm{\beta} \in \Sigma^{n_0}} F_1(\h, \bm{\beta}, \epsilon, n)  = \mathcal{O}_K(\h^{N})
\end{align*}
in the given parameter range with $\h$ possibly smaller.  Finally, 
\begin{align} \label{eqn:F_1_ub}
|F_1(\h, \bm{\beta}, \epsilon, n)| \leq 2 \| |\Pi_{\bm{\beta}}|^2 \Psi_{\h}\| \cdot \frac{\epsilon n}{|\log \h|} +  \left( \frac{\epsilon n}{|\log \h|}  \right)^2.
\end{align}
A similar formula, with analogous errors in $\h$, holds for the ``backwards" quantities:
\begin{align*}
\sum_{\bm{\alpha} \in \Sigma^n} \left\| \Pi_{\bm{\beta}}(-n)  \left[ \Pi_{\alpha_{-n}}(-n) \dots \Pi_{\alpha_{-1}}(-1) \right] \Psi_{\h} \right\|^2 = \langle |\Pi_{\bm{\beta}}|^2 U^{-n} \Psi_{\h}, U^{-n} \Psi_{\h} \rangle + \tilde{G}
\end{align*}
with the formula equating this to $\tilde{\mu}_{\h}(\bm{\beta})$ having the defect $\tilde{F}_1(\bullet)$ and errors of the same magnitude in $\h$.  
\end{prop}

Our proposition should be compared with the shift-invariance statement \cite[Proposition 1.3.1]{An08} and \cite[Lemma 3.14]{Non12}.  


\begin{proof}

We proceed almost exactly the same as in \cite[Lemma 3.14]{Non12} until the last few steps.   We do the calculation for the forward measure as that for the backwards measure is nearly identical. We have
\begin{align}
\nonumber & \sum_{\bm{\alpha} \in \Sigma^n} \langle |\Pi_{\bm{\beta}}(n)|^2 \Pi_{\bm{\alpha}} \Psi_{\h}, \Pi_{\bm{\alpha}} \Psi_{\h}  \rangle  = \sum_{ \bm{\alpha'} \in \Sigma^{n-1}}   \sum_{\alpha_{n-1} \in \Sigma}  \left(  \langle \Pi_{\alpha_{n-1}}(n-1)^* |\Pi_{\bm{\beta}}(n)|^2 \Pi_{\alpha_{n-1}}(n-1) \Pi_{\bm{\alpha}'} \Psi_{\h}, \Pi_{\bm{\alpha}'} \Psi_{\h}  \rangle  \right) \\
\nonumber & = \sum_{\bm{\alpha}' \in \Sigma^{n-1}} \sum_{\alpha_{n-1} \in \Sigma} \langle \Pi_{\alpha_{n-1}}(n-1)^* \Pi_{\alpha_{n-1}}(n-1) |\Pi_{\bm{\beta}}(n)|^2 \Pi_{\bm{\alpha}'} \Psi_{\h}, \Pi_{\bm{\alpha}'} \Psi_{\h}  \rangle  \\
\label{eqn:shift_calc_1} & +  \sum_{\bm{\alpha}' \in \Sigma^{n-1}} \sum_{\alpha_{n-1} \in \Sigma} \langle \Pi_{\alpha_{n-1}}(n-1)^* \left[ |\Pi_{\bm{\beta}}(n)|^2, \Pi_{\alpha_{n-1}}(n-1)  \right]\Pi_{\bm{\alpha}'} \Psi_{\h}, \Pi_{\bm{\alpha}'} \Psi_{\h}  \rangle  \\
\nonumber & = \left( \sum_{\bm{\alpha}' \in \Sigma^{n-1}} \langle |\Pi_{\bm{\beta}}(n)|^2 \Pi_{\bm{\alpha}'} \Psi_{\h}, \Pi_{\bm{\alpha}'} \Psi_{\h}  \rangle \right) \\
\label{eqn:shift_calc_2} & + \left( \mathcal{O}(\h^{\delta_{Ehr}}) \cdot \sum_{\alpha_{n-1} \in \Sigma^1} \| \Pi_{\alpha_{n-1}}(n-1) \| \right) + K^{n-1} \cdot \mathcal{O} \left( \h^{N} \right) + K  \cdot \mathcal{O}(\h^{N})
\end{align}
The error terms in the final line are deduced as follows in the next two paragraphs. 

We start from the first term in equation (\ref{eqn:shift_calc_1}). We use that $\sum_{\alpha_{n-1}}  \Pi_{\alpha_{n-1}}(n-1) \Pi_{\alpha_{n-1}}(n-1)^*  = \tilde{S}_n+ R_{\h}$, thanks to the proof of Proposition \ref{prop:pou_quasi} found in \cite[Proposition 3.1]{Non12} and $\tilde{S}_n$ satisfying the same properties as $S_n$. Apply the microlocalization statement Lemma \ref{lem:strong_micro}. 
 Proposition \ref{p:strong_microlocal_quasi} shows $\Pi_{\bm{\beta}}(n) \Pi_{\bm{\alpha}'} \Psi_{\h}$ continues to be microlocalized near $\mathcal{E}$ at scale $\hbar^{\gamma}$.  In (\ref{eqn:shift_calc_2}), this explains the first and third terms of $\sum_{\bm{\alpha}' \in \Sigma^{n-1}} \langle |\Pi_{\bm{\beta}}(n)|^2 \Pi_{\bm{\alpha}'} \Psi_{\h}, \Pi_{\bm{\alpha}'} \Psi_{\h}  \rangle$ and $K^{n-1} \cdot \mathcal{O}(\h^{N})$, respectively.

We now work from the second term in (\ref{eqn:shift_calc_1}) to the second and fourth terms in (\ref{eqn:shift_calc_2}).  The assumption that $n + n_0 \leq \lfloor 2 T_{\delta_{Ehr}, \kappa_{slab}, \h} \rfloor$ is crucial in estimating our commutator.  Note that 
\begin{align*}
& \left[ |\Pi_{\bm{\beta}}(n)|^2, \Pi_{\alpha_{n-1}}(n-1)  \right] \\
& = U^{-n/2} \left( |\Pi_{\bm{\beta}}|^2(\frac{n}{2}) U^{n/2} U^{-n/2} \Pi_{\alpha_{n-1}}(\frac{n}{2}-1) - \Pi_{\alpha_{n-1}}(\frac{n}{2}-1) U^{n/2} U^{-n/2}|\Pi_{\bm{\beta}}|^2(\frac{n}{2}) \right) U^{n/2}.
\end{align*} The estimate on this commutator will follow from an application of the ``long-time" form of Egorov's Theorem \cite[Proposition 5.1]{Non12} on the operator 
\begin{align*}
|\Pi_{\bm{\beta}}|^2(\frac{n}{2}) \Pi_{\alpha_{n-1}}(\frac{n}{2}-1) - \Pi_{\alpha_{n-1}}(\frac{n}{2}-1)|\Pi_{\bm{\beta}}|^2(\frac{n}{2}),
\end{align*}  
and the invariance of norms under conjugation by $U^{-n/2}$.  We perform the following steps: 
sum first in $\alpha_{n-1}$ in order to bound above by 
\begin{equation*}
    \sum_{\bm{\alpha}'} \sum_{\alpha_{n-1}} \| \Pi_{\alpha_{n-1}} \| \cdot   \| \left[ |\Pi_{\bm{\beta}}(n)|^2, \Pi_{\alpha_{n-1}}(n-1)  \right] \|  \cdot \|\Pi_{\bm{\alpha}'} \Psi_{\h} \|^2 
\end{equation*}
Now, use $\| \left[ |\Pi_{\bm{\beta}}(n)|^2, \Pi_{\alpha_{n-1}}(n-1)  \right] \| = \mathcal{O}(\h^{\delta_{Ehr}})$ and apply Lemma \ref{lem:compatible_meas} in $\bm{\alpha}'$. This gives the second and fourth terms in (\ref{eqn:shift_calc_2}).

We can repeat this procedure $n-1$ more times to obtain that $\sum_{\bm{\alpha} \in \Sigma^n} \langle |\Pi_{\bm{\beta}}(n)|^2 \Pi_{\bm{\alpha}} \Psi_{\h}, \Pi_{\bm{\alpha}} \Psi_{\h}  \rangle $ equals
\begin{align*}
& \langle |\Pi_{\bm{\beta}}|^2 U^{n} \Psi_{\h}, U^{n} \Psi_{\h} \rangle \\
& + \underbrace{n \cdot \left( \mathcal{O}(\h^{\delta_{Ehr}}) \cdot \sum_{\alpha \in \Sigma^1} \| \Pi_{\alpha} \| \right)+  \left( \sum_{j=1}^{n-1} K^j \right)  \cdot \mathcal{O} \left( \h^{N} \right) +  n \cdot \mathcal{O}_{K} \left( \h^{N} \right)}_{=: G(\h, \mathbf{\beta}, \epsilon, n)}.
\end{align*}
Note that 
\begin{align} \label{eqn:measure_compare}
\| \Pi_{\bm{\beta}} U^n \Psi_{\h} \|^2 + G(\h, \bm{\beta}, n)= \| \Pi_{\bm{\beta}}\Psi_{\h} \|^2+  F_1(\h, \bm{\beta}, \epsilon, n) + G(\h, \bm{\beta}, \epsilon, n).
\end{align}
The equation $\sum_{\bm{\beta}} F_1(\h, \bm{\beta}, \epsilon, n) = \mathcal{O}(\h^{\infty})$, as $\h \rightarrow 0$, now follows immediately from Proposition \ref{prop:pou_quasi} applied to both sides of (\ref{eqn:measure_compare}), Lemma \ref{lem:strong_micro} applied to $U^n \Psi_{\h}$ as $U^n$ is unitary, and the fact that our states are normalized.

Lastly, to give a bound on $F_1$, we use the simple operator formula of 
\begin{align*}
U^n - e^{-in/2\h} Id = \frac{-i}{\h} \int_0^n \left(P(\h) - \frac{1}{2} \right) U^{s} e^{is/2\h} e^{-in/2\h} ds
\end{align*}
to arrive at equation (\ref{eqn:F_1}).  We remind ourselves that $\sum_{\bm{\beta}} \langle |\Pi_{\bm{\beta}}|^2 \Phi_{\h}, \Phi_{\h}  \rangle = \langle \Phi_{\h}, \Phi_{\h}  \rangle \leq \left( \frac{ \epsilon n}{|\log \h|} \right)^2 + \mathcal{O}(\h^{\infty})$ with each summand itself being real and positive, therefore giving us (\ref{eqn:F_1_ub}).
\end{proof}

\subsection{A remark on the continuity of $F_1$} \label{subsect:assumption_A}

Consider Proposition \ref{l:approx_shift}.  Suppose that given $\overline{\epsilon}>0$, there exists a possibly smaller $\h_0$ and a $\overline{\delta}(\overline{\epsilon})>0$ such that for all quasiprojectors $\Pi_{\bm{\beta}}$ where $\mu_{\h}(\bm{\beta}) < \overline{\delta}$ when $\h \leq \h_0$ and $\bm{\beta} \in \Sigma^{n_0}$, we have that 
\begin{equation*}
\sup_{n \leq  \lfloor 2T_{\delta_{Ehr}, \kappa_{slab}, \h} \rfloor - n_0 } |F_1(\h, \bm{\beta}, \epsilon, n)| < \overline{\epsilon}.
\end{equation*}

This statement is trivial when $\epsilon = o(1)$ in $\h$.  Moreover, it is simply the limiting definition of absolute continuity of measures.  Such a statement appears to require a knowledge of the asymptotics of $\mu_{\h}$ away from the support of $\mu_{sc}$, but contained in $\mathcal{E}$, which the author was unable to prove. A previous version of this manuscript assumed this continuity-type condition on $F_1$.  The author was later informed by St\'ephane Nonnenmacher that the assumption fails for the toy model of the quantum cat map, making a stronger case to state our Theorem \ref{thm:main} without additional assumptions.

\section{Quantum entropies and the entropic uncertainty principle} \label{sect:quant_entropies_EUP}

\subsection{Definitions: Quantum entropies and pressures}
As mentioned at the top of Section \ref{sect:quantum_measures}, we suggest the readers who are not fluent with the material found in \cite{AN07, AKN06} to read Appendix \ref{sect:entropy_quasiproj} and familiarize themselves with the notions of classical metric entropy of a measure with respect to a partition $H_0^{n-1}(\bullet)$, entropy of a measure with respect to a smooth partition $H_0^{n-1}(\bullet)$ (a convenient abuse of notation), and classical/quantum pressure functionals $p^{n-1}_0(\bullet)$.  We give a brief introduction to the quantum versions of entropy and pressure now:
\begin{defn}
Given an $L^2$-normalized quantum state $\Psi_{\h}$, a smooth quantum partition $\mathcal{P}_{\rm{sm}, \rm{q}}$, and an associated set of positive weights (see $\bm{w}$ (\ref{eqn:pressure_weights})), we define the \textit{quantum entropy} as
\begin{equation*}
H_0^{n-1}(\Psi_{\h}, \mathcal{P}_{\rm{sm, q}} ) :=  \sum_{\bm{\alpha} \in \Sigma^n} \eta \left( \|\Pi_{\bm{\alpha}} \Psi_{\h} \|^2 \right) 
\end{equation*}
and the \textit{quantum pressure} as
\begin{equation*}
p_0^{n-1}(\Psi_{\h}, \mathcal{P}_{\rm{sm,q}}, \bm{w}) :=  \sum_{\bm{\alpha} \in \Sigma^n} \eta \left( \|\Pi_{\bm{\alpha}} \Psi_{\h} \|^2 \right) -  \sum_{\bm{\alpha} \in \Sigma^n} \| \Pi_{\bm{\alpha}} \Psi_{\h} \|^2 \log \left( \bm{w}_{\bm{\alpha}}^2 \right). 
\end{equation*} 
\end{defn}
We suppress the mention of $\mathcal{P}_{\rm{sm,q}}$ in our notation of quantum entropies/pressures until the end of our argument when both the smoothing parameter $\eta$ goes to 0 (see Definition \ref{def:smooth_partition}) and $K$ goes to $\infty$.

\subsection{Entropic uncertainty principle} \label{sect:EUP}
The following proposition is a recording of \cite[Proposition 3.12]{Non12} and plays a major role in the proof of Theorem \ref{thm:main}. We keep the more abstract form of the statement for a cleaner exposition.  Towards the end of this section is it explained why we chose to work with spectrally localized modes rather than more general ones of logarithmic width.  

\begin{prop}[Entropic uncertainty principle for microlocal weighted partitions] \label{prop:EUP}
On the Hilbert space $\mathcal{H} = L^2(M, dV_g)$, consider two approximate quantum partitions (of unity), that is two finite sets of bounded operators $\bm{\rho} = \{ \rho_i, \, i \in I \}$ and $\bm{\tau} = \{ \tau_j, \, j \in J \}$ satisfying the identities
\begin{equation*}
\sum_{i \in I} \rho_i^* \rho_i = S_{\bm{\rho}}, \, \, \, \sum_{j \in J} \tau_j^* \tau_j = S_{\bm{\tau}},
\end{equation*}
and two families of weights $\bm{v} = \{v_i, i \in I \}$, $\bm{w} = \{ w_j, j \in J \}$ satisfying $V^{-1} \leq v_i, w_j \leq V$ for some $V \geq 1$.

Assume that for some $0 \leq \mathscr{R} \leq \min \{|I|^{-2} V^{-2}, |J|^{-2} V^{-2} \}$, the above sum of operators satisfies $0 \leq \|S_{\bm{\rho}/\bm{\tau}} \|_{op} \leq 1 + \mathscr{R}$. Furthermore, let $S_{c_1}, S_{c_2}$ be two Hermitian operators on $\mathcal{H}$ satisfying $0 \leq \|S_{c_1/c_2} \|_{op} \leq 1 + \mathscr{R}$ and the relations
\begin{align*}
& \| \left( S_{c_2} - \rm{Id}) \rho_i \, S_{c_1} \right) \|_{op} \leq \mathscr{R}, \, \mbox{ for } i \in I \\
& \| \left( S_{\bm{\rho}/\bm{\tau}} - \rm{Id}) \, S_{c_1} \right) \|_{op} \leq \mathscr{R}, \mbox{ and }
\end{align*}
Set 
\begin{equation*}
c_{\rm{cone}} := \max_{i \in I, j \in J} v_i w_j \|\tau_j \rho_i^*  S_{c_2} \|_{op}.
\end{equation*}

For any $\Psi \in \mathcal{H}$ normalized such that $\| (\rm{Id} - S_{c_1} ) \Psi \|_{op} \leq \mathscr{R}$, the quantum pressures with respect to the weighted partitions $(\bm{\rho}, \bm{v})$ and $(\bm{\tau}, \bm{w})$ satisfy the bound
\begin{equation*}
p(\Psi, \bm{\rho}, \bm{v}) + p(\Psi, \bm{\tau}, \bm{w}) \geq -2 \log \left( c_{\rm{cone}} + 3 |I| V^2 \mathscr{R} \right) + \mathcal{O}\left( \mathscr{R}^{1/5} \right)
\end{equation*}
with the implicit constant in the big-O notation being independent of the weighted partitions or cutoff operators $S_{c_1/c_2}$.
\end{prop}

The formula for $c_{\rm{cone}}$ determines our choice of quantum partitions and weights; in particular, we need that $\tau_j \rho_i ^*S_{c_2}$ be equal to $\Pi_{\bm{\alpha}} U^n \Pi_{\bm{\beta}} U^{-n} \chi^{(n)}(P(\h)-\frac{1}{2})$ (see Proposition \ref{prop:hyp_disp}) where $|\bm{\alpha}| = |\bm{\beta}| = n$ and $n = \lfloor 2 T_{\delta_{Ehr}, \kappa_{slab}, \h} \rfloor$.  The following list sets the remaining choices in place:
\begin{align}\label{eqn:EUP_quantities}
\begin{cases}
& \bm{\tau} = \{ \Pi_{\bm{\alpha}} \, | \, |\bm{\alpha}|=n \}  \\
& \bm{\rho} = \{ U^{n} \Pi_{\bm{\beta}}^* U^{-n} \, | \,  |\bm{\beta}|=n \} \\
& \bm{v} = \bm{w} = \{ w_{\bm{\beta}}, v_{\bm{\beta}} := \sqrt{J_n^u(\bm{\beta})} \, | \, |\bm{\beta}|=n \} \\
& S_{c_1} = \chi^{(0)}(P(\h)-\frac{1}{2}), \, S_{c_2} = \chi^{(n)}(P(\h)-\frac{1}{2}) 
\end{cases}
\end{align}
The formulas for the coarse-grained Jacobians in  (\ref{eqn:coarse_grain_Jac}) show $\sqrt{J_{2n}^u(\bm{\alpha}\bm{\beta})} = \sqrt{J_n^u(\bm{\beta})} \sqrt{J_n^u(\bm{\alpha})}$. Moreover, $|I|=|J|=K^n$ and there exists a uniform $\tilde{C}>0$ such that 
\begin{equation} \label{eqn:Jac_bounds}
\tilde{C}^{-1} e^{n \Lambda^u_{\rm{min}}(d-1)/2} \leq \sqrt{J_n^u(\bm{\beta})} \leq \tilde{C} e^{n \lambda_{\rm{max}}(d-1)/2}
\end{equation}
where $\Lambda^u_{\rm{min}}$ is the minimal unstable expansion rate on $\mathcal{E}$ \cite[Page 18]{Non12}, \cite[Equation (2.17)]{AN07}.  Thanks to Proposition \ref{prop:hyp_disp}, 
\begin{equation*}
c_{\rm{cone}} \leq C h^{-(d-1 + c_0 \delta_{energy})/2}
\end{equation*}
for all $\h$ sufficiently small.  For our $\epsilon$-logarithmic modes $\Psi = \Psi_{\h}$, we can set $\mathscr{R}= \h^N$ for $N$ sufficiently large thanks to our microlocalization property Lemma \ref{lem:strong_micro} and Proposition \ref{p:strong_microlocal_quasi}.  Therefore, for $\h \leq \h_0 \left( \mathcal{P}_{\rm{sm}}, \delta_{Ehr}, \kappa_{slab}, N \right)$ possibly smaller than before, we have established a bound of
\begin{equation*}
p_0^{n-1}(\Psi_{\h}, \mathcal{P}_{\rm{sm,q}}, \bm{v}) + p_{-n}^{-1}(\Psi_{\h}, \mathcal{P}_{\rm{sm,q}}^*, \bm{w}) \geq -(d-1 + c_0 \delta_{Ehr})|\log \h| - 2 \log 10 C + \mathcal{O}(\h^{N/4})
\end{equation*}
where $\mathcal{P}_{\rm{sm}}^*$ is the set of adjoints of operators in $\mathcal{P}_{\rm{sm}}$ and the implicit constant is independent of $n$.

\begin{rem}
Here, we see the convenience of spectrally localized modes: to have $\mathscr{R}$ small enough to dominate the growth of $K^n \cdot V^2=\mathcal{O}(\h^{-(d-1)-\frac{\log K}{2 \lambda_{max}}})$, as dictated by the choice of weights $w_{\bullet}, v_{\bullet}$, and make the corresponding product smaller than $c_{\rm{cone}}$.  Were $\{\Psi_{\h}\}_{\h}$ a more general logarithmic family, then $\mathscr{R} = \mathcal{O}\left(\frac{\h}{|\log \h|}\right)$ would be our best current estimate, which is insufficient.
\end{rem}

\section{Sub-additivity statements} \label{sect:subadd}

Following the strategy of Anantharaman-Nonnenmacher \cite[Section 4]{AN07} to establish the sub-additivity of quantum entropies, we aim to keep a better track of certain quantities in their estimates which arise from our $\epsilon$-logarithmic modes. 

\subsection{Sub-additivity of quantum entropies}

\begin{lem}[Approximate sub-additivity of quantum entropies] \label{lem:subadd_entropy}
Fix some $n_0 \geq 0$, $\delta_{Ehr} \in (0,1)$, and $\kappa_{slab}$. Let $\gamma>0$.  

Given an $\epsilon$-logarithmic family $\{\Psi_{\h}\}_{\h}$, there exists $\h_0(\epsilon, K, n_0, \delta_{Ehr}, \kappa_{slab},  \gamma)$ such that for all $\h \leq \h_0$, any $\bm{\beta} \in \Sigma^{n_0}$, and  integers $n$ in the range $[0, 2 T_{\delta_{Ehr}, \kappa_{slab}, \h} - n_0]$ we have the estimate 
\begin{align*}
H_0^{n + n_0-1}(\Psi_{\h}) \leq   H_0^{n-1}(\Psi_{\h}) + H_{0}^{ n_0-1}(U^{n}\Psi_{\h}) +  R_1\left(\h, n, n_0, \epsilon \right)  + \mathcal{O}_{K, n_0}(\h^{\gamma})
\end{align*} 
where
\begin{align} \label{eqn:sub_add_error}
\nonumber R_1(\h, n, n_0, \epsilon) {}:={} & - \sum_{|\bm{\beta}| = n_0} \left( \| \Pi_{\bm{\beta}}U^{n} \Psi_{\h} \|^2 + G(\h, \bf{\beta}, \epsilon, n) \right) \log \left( \| \Pi_{\bm{\beta}} U^n \Psi_{\h} \|^2 + G(\h, \bm{\beta}, \epsilon, n) ) \right)  \\
\nonumber &  - H^{n_0-1}_0 (U^n \Psi_{\h} )
\\ 
& = \mathcal{O}_{K, n_0} \left( \eta(\h^{\gamma}) \right).
\end{align}
Here, $G$ is as in Proposition \ref{l:approx_shift}.  A similar formula holds for $H_{-1}^{-n-n_0}(\Psi_{\h})$.
\end{lem}

\begin{proof}
First, we notice the following simple implication from Proposition \ref{l:approx_shift}.  If $\| \Pi_{\bm{\beta}} U^n \Psi_{\h} \|^2 + G(\h, \bm{\beta}, \epsilon, n) = 0$, then $\| \Pi_{\bm{\beta}}(n) \Pi_{\bm{\alpha}} \Psi_{\h} \|^2 =0$ for all $\bm{\alpha} \in \Sigma^n.$
Using this, we write
\begin{align*}
 H_0^{n+n_0-1}(\Psi_{\h}) = & -\sum_{|\bm{\alpha}|=n} \sum_{|\bm{\beta}|=n_0} \| \Pi_{\bm{\beta}}(n) \Pi_{\bm{\alpha}} \Psi_{\h} \|^2 \, \times \\
 & \log \left( \frac{\| \Pi_{\bm{\beta}}(n) \Pi_{\bm{\alpha}} \Psi_{\h} \|^2}{ \| \Pi_{\bm{\beta}}U^{n} \Psi_{\h} \|^2 + G(\h, \bm{\beta}, \epsilon, n)}  \cdot \left( \| \Pi_{\bm{\beta}}U^{n} \Psi_{\h} \|^2 + G(\h, \bm{\beta}, \epsilon, n) \right) \right).
\end{align*}
Using the homomorphism property of logarithms, after again rewriting 1 as a quotient, we arrive at
\begin{align}
\nonumber & -\sum_{|\bm{\alpha}|=n} \sum_{|\bm{\beta}|=n_0}  \left( \| \Pi_{\bm{\beta}}U^{n} \Psi_{\h} \|^2 + G(\h, \bm{\beta}, \epsilon, n)  \right) \frac{\| \Pi_{\bm{\beta}}(n) \Pi_{\bm{\alpha}} \Psi_{\h} \|^2}{ \| \Pi_{\bm{\beta}}U^{n} \Psi_{\h} \|^2 + G(\h, \bm{\beta}, \epsilon, n)  } \times \\
\nonumber &  \log \left( \frac{\| \Pi_{\bm{\beta}}(n) \Pi_{\bm{\alpha}} \Psi_{\h} \|^2}{\| \Pi_{\bm{\beta}}U^{n} \Psi_{\h} \|^2 + G(\h, \bm{\beta}, \epsilon, n)  } \right) \\
\label{eqn:ent_recond} & -\sum_{|\bm{\alpha}|=n} \sum_{|\bm{\beta}|=n_0} \| \Pi_{\bm{\beta}}(n) \Pi_{\bm{\alpha}} \Psi_{\h} \|^2 \, \log \left( \| \Pi_{\bm{\beta}}U^{n} \Psi_{\h} \|^2 + G(\h, \bm{\beta}, \epsilon, n)  \right)
\end{align}
for all $\h \leq \h_0$.

First, we keep in mind 
\begin{equation*}
\sum_{|\bm{\beta}| = n_0} \left( \| \Pi_{\bm{\beta}}U^{n} \Psi_{\h} \|^2 + G(\h, \bm{\beta}, \epsilon, n)  \right) = 1 + \mathcal{O}_{K,n_0}(\h^{3 \gamma})
\end{equation*}
thanks to the shift-invariance statement in Proposition \ref{l:approx_shift} and the estimate on $G$, for $\h_0$ possibly smaller.  Thus there exists $f = \mathcal{O}_{K,n_0}(\h^{3 \gamma})$ such that
\begin{equation*}
\sum_{|\bm{\beta}| = n_0} \left( \| \Pi_{\bm{\beta}}U^{n} \Psi_{\h} \|^2 + G(\h, \bm{\beta}, \epsilon, n)  + f(\h, \bm{\beta}, \epsilon, n) \right)  = 1.
\end{equation*}
On the first summation in (\ref{eqn:ent_recond}), we use that $\eta(\bullet)$ is a convex function and apply Jensen's Inequality in $\bm{\beta}$, resulting in it being less than or equal to 
\begin{align*}
& - \sum_{|\bm{\alpha}|=n} \eta \left( \sum_{|\bm{\beta}|=n_0} \| \Pi_{\bm{\beta}}(n) \Pi_{\bm{\alpha}} \Psi_{\h} \|^2 \left(1 + \mathcal{O}_{K,n_0}(\h^{3 \gamma}) \right)  \right) + \mathcal{O}_{K,n_0}(\h^{3 \gamma/2}) \\
& = - \sum_{|\bm{\alpha}|=n} \eta \left(\| \Pi_{\bm{\alpha}} \Psi_{\h} \|^2 \right) + \mathcal{O}_{K,n_0}(\eta(\h^{3 \gamma/2}))
\end{align*}
after using Proposition \ref{prop:pou_quasi}, the sublinearity inequality $|\eta(s+s') - \eta(s)| \leq \eta(|s'|) $, and taking $\h_0$ possibly smaller. For the second summation, we sum in $\bm{\alpha}$ first as in the proof of Lemma \ref{l:approx_shift}, leading us to an equality with
\begin{align*}
& - \sum_{|\bm{\beta}| = n_0} \left( \| \Pi_{\bm{\beta}}U^{n} \Psi_{\h} \|^2 + G(\h, \bm{\beta}, \epsilon, n) \right) \log \left( \| \Pi_{\bm{\beta}} U^n \Psi_{\h} \|^2 + G(\h, \bm{\beta}, \epsilon, n)  \right)  \\
& = H_0^{n_0-1}(U^n \Psi_{\h} ) + R_1(\h, n, n_0, \epsilon).
\end{align*}
It follows that $R_1 = \mathcal{O}_{K,n_0} \left( \eta(\h^{\gamma}) \right)$, after again applying our sublinearity inequality to $\eta$, for $\h_0$ possibly smaller.
\end{proof}

We immediately have the following:
\begin{cor}
Consider the hypotheses of Lemma \ref{lem:subadd_entropy} and set $n = (q-1)n_0$ for $q \in \mathbb{N}$.  Then there exists $\h_0$ such that for all $\h \leq \h_0$,
\begin{align*}
H^{qn_0 - 1}_0(\Psi_{\h}) \leq  H^{n_0 - 1}_0(\Psi_{\h}) +  \sum_{j=1}^{q-1} H^{n_0 - 1}_0(U^{(q-j)n_0}\Psi_{\h}) + \sum_{j=1}^{q-1} R_1(\h, j n_0, n_0, \epsilon) + \mathcal{O}_{K,n_0}(\h^{\gamma}).
\end{align*}
\end{cor}

\subsection{Sub-additivity of quantum pressures}

\begin{lem}[``New" potentials] \label{lem:cont_pressure_potential}
Consider the weights in (\ref{eqn:EUP_quantities}).  Then under the same hypotheses as Lemma \ref{lem:subadd_entropy} and (a possibly smaller) value of $\h_0$, we have for all $\h \leq \h_0$,
\begin{align*}
& - \sum_{|\bm{\alpha}|=n} \sum_{|\bm{\beta}|=n_0} 2 \left\|  \Pi_{\bm{\beta}}(n)\Pi_{\bm{\alpha}} \Psi_{\h} \right\|^2 \log (w_{\bm{\alpha \beta}})  = \\
& - \sum_{|\bm{\alpha}|=n}  2 \mu_{\h}( \bm{\alpha} ) \log (w_{\bm{\alpha}}) - \sum_{|\bm{\beta}|=n_0} 2 \mu_{\h}( \bm{ \beta} ) \log (w_{\bm{\beta}}) + F_2(\h,\epsilon, n, n_0,  \bm{w}) + \mathcal{O}_{K,n_0}(h^{\epsilon'})
\end{align*}
where
\begin{align} \label{eqn:new_potentials}
F_2(\h, \epsilon, n, n_0, \bm{w}) & := -  \sum_{|\bm{\beta}|=n_0} 2 \,  \log (w_{\bm{\beta}}) F_1(\h,\bm{\beta}, \epsilon, n). 
\end{align}
A similar formula holds for 
\begin{align*}
-\sum_{|\bm{\alpha}|=n} \sum_{|\bm{\beta}|=n_0} 2 \left\| \Pi_{\bm{\beta}}(-n)  \left[ \Pi_{\alpha_{-n}}(-n) \dots \Pi_{\alpha_{-1}}(-1) \right] \Psi_{\h} \right\|^2 \log(w_{\bm{\alpha} \bm{\beta}})
\end{align*}
and the backwards measure $\tilde{\mu}_{\h}$ with an analogously defined defect $\tilde{F}_2(\bullet)$.
\end{lem}

\begin{proof}
Notice  $\sum_{|\bm{\alpha}| = n} \sum_{|\bm{\beta}| = n_0} \mu_{\h}(\bm{\alpha \beta}) = \sum_{|\bm{\alpha}| = n} \mu_{\h}(\bm{\alpha}) + \mathcal{O}_{K,n_0}(\h^{\infty})$ thanks to Proposition \ref{prop:pou_quasi} and Proposition \ref{p:strong_microlocal_quasi}.  We will use this estimate shortly.

Now consider $-\log(w_{\bm{\alpha}}) - \log(w_{\bm{\beta}})$, multiply this by $\mu_{\h}(\bm{\alpha \beta})$, and sum over $\bm{\alpha}, \bm{\beta}$.  On the sum with weight $-\log(w_{\bm{\alpha}})$, sum first in $\bm{\beta}$ and use the previous estimate.  On the sum with weight $-\log(w_{\bm{\beta}})$, sum first in $\bm{\alpha}$ and apply Proposition \ref{l:approx_shift} to give
\begin{align*}
- \sum_{|\bm{\beta}|=n_0} 2 \log(w_{\bm{\beta}}) \, \mu_{\h}(\bm{\beta}) - \sum_{|\bm{\beta}|=n_0} 2 \,  \log (w_{\bm{\beta}}) \, F_1(\h, \bm{\beta}, \epsilon, n) - \sum_{|\bm{\beta}|=n_0} 2 \,  \log (w_{\bm{\beta}}) \, G(\h, \bm{\beta}, \epsilon, n)
\end{align*}
and in turn the proposed formula. 
\end{proof}

A combination of Lemmas \ref{lem:cont_pressure_potential} and \ref{lem:subadd_entropy} lead us to:

\begin{cor}\label{c:subadd}[Approximate sub-additivity of quantum pressure functional]

Fix some $n_0 \geq 0$, $\delta_{Ehr} \in (0,1)$, $\kappa_{slab}>0$ and $\gamma>0$. Consider $R_1(\bullet)$ as in Lemma \ref{lem:subadd_entropy} and $F_2(\bullet)$ as in Lemma \ref{lem:cont_pressure_potential}. Given an $\epsilon$-logarithmic family $\{\Psi_{\h}\}_{\h}$, there exists $\h_0(\epsilon, n_0, \delta_{Ehr}, \kappa_{slab}, K, \gamma)$ such that for all $\h \leq \h_0$  and $n+n_0 \leq \lfloor 2 T_{\delta_{Ehr}, \kappa_{slab}, \h} \rfloor$, the (forward) quantum pressure with associated weights $\textbf{w}$ satisfies
\begin{align*}
p_0^{n + n_0-1}(\Psi_{\h}, \mathcal{P}_{\rm{sm}}, \textbf{w}) & \leq  H_0^{n-1}(\Psi_{\h}) + H_0^{n_0-1}(U^n\Psi_{\h})  \\
& - \sum_{|\bm{\alpha}|=n}  2 \mu_{\h}( \bm{\alpha} ) \log (w_{\bm{\alpha}})  - \sum_{|\bm{\beta}|=n_0} 2 \mu_{\h}(  \bm{ \beta} ) \log (w_{\bm{\beta}})  + F_2(\h,\epsilon, n, n_0,  \bm{w}) \\
& + R_1(\h, n, n_0, \epsilon) +  \mathcal{O}_{n_0,K} (\h^{\gamma}).
\end{align*}
A similar inequality holds for the (backwards) quantum pressure $p_{-(n+n_0)}^{-1}(\Psi_{\h}, \mathcal{P}_{\rm{sm}}^*, \bm{v})$.
\end{cor}

Our various statements in this section, combined with Proposition \ref{prop:EUP} (entropic uncertainty principle), culminate in the main proposition of Section \ref{sect:subadd}:

\begin{prop}[Penultimate pressure bound at ``finite time" (i.e. $\h$-independent) $n_0$] \label{prop:penult_bound}
Consider an $\epsilon$-logarithmic family $\{\Psi_{\h}\}_{\h}$.  Let $\delta_{Ehr} \in (0,1)$, $\kappa_{slab}>0$, and $\gamma>0$.  Fix some $n_0 > 0$. For $\h < \h_0$ as returned from Corollary \ref{c:subadd}, split the time $n = \lfloor 2 T_{\delta_{Ehr}, \kappa_{slab}, \h} \rfloor$ into $n = qn_0 + r$ for $r \in [0, n_0) \cap \mathbb{N}$.  

Then for $c_0$ appearing in Proposition \ref{prop:hyp_disp} and a possibly smaller value of $\h_0$, the entropies satisfy the lower bound for all $\h \leq \h_0$:
\begin{align} \label{eqn:penutl_bound}
\nonumber & H^{n_0 - 1}_0(\Psi_{\h}) + H^{-1}_{-n_0}(\Psi_{\h})  + \sum_{j=1}^{q-1} H^{n_0 - 1}_0(U^{(q-j)n_0}\Psi_{\h}) + \sum_{j=1}^{q-1} H^{- 1}_{-n_0} (U^{-(q-j)n_0}\Psi_{\h}) \\
\nonumber & - q\sum_{|\bm{\beta}|=n_0} 2 \mu_{\h}( \bm{ \beta} ) \log (w_{\bm{\beta}}) - q\sum_{|\bm{\beta}|=n_0} 2 \tilde{\mu}_{\h}( \bm{ \beta} ) \log (v_{\bm{\beta}}) \\
\nonumber &  + \sum_{j=1}^{q-1} F_2(\h,\epsilon, j n_0, n_0, \textbf{w}) + \sum_{j=1}^{q-1} \tilde{F}_2(\h, \epsilon, j n_0, n_0, \textbf{v})  \\
\nonumber & + \sum_{j=1}^{q-1} R_1(\h, jn_0, n_0, \epsilon) + \sum_{j=1}^{q-1} \tilde{R}_1(\h, jn_0, n_0, \epsilon)  + F_2(\h,\epsilon, q n_0, r, \textbf{w}) + \tilde{F}_2(\h, \epsilon, q n_0, r ,\textbf{v}) \\
\nonumber & + p_0^{r-1}(U^{qn_0}\Psi_{\h},\textbf{w}) + p_{-r}^{-1}(U^{-qn_0}\Psi_{\h}, \textbf{v}) + \mathcal{O}_{K,n_0} (\h^{\gamma})\\
& \geq   -(d-1 + c_0\delta_{Ehr}) \times |\log \h|. 
\end{align}
\end{prop}

\begin{proof}
Apply the entropic uncertainty Proposition \ref{prop:EUP}.  Next, use Corollary \ref{c:subadd} and Lemma \ref{lem:cont_pressure_potential} $q-1$ times. 
\end{proof}

\section{Comparing $H^{n_0 - 1}_0(U^{(q-j)n_0}\Psi_{\h})$ to $H^{n_0 - 1}_0(\Psi_{\h})$} \label{sect:defect}

This section forms the primary extension of \cite{AKN06, AN07}: it analyses further the entropy of states $U^{j' n_0} \Psi_{\h}$, where $j' \leq \lfloor 2 T_{\delta_{Ehr}, \kappa_{slab}, \h} \rfloor$, at finite times $n_0$.  Recall that in \cite{AN07, AKN06} for the case of eigenfunctions, $q^{-1} \sum_{j} F_1(\bullet)$ vanishes in the semiclassical limit with $\sum_{j=1}^{q-1} H^{n_0 - 1}_0(U^{(q-j)n_0}\Psi_{\h}) = (q-1) H_0^{n_0-1}(\Psi_{\h})$; this is not the case for $\epsilon$-logarithmic modes.  This section is dedicated to analysing further $q^{-1} \sum_{j} F_1(\bullet)$, which itself will feed into $q^{-1} \sum_j F_2 (\bullet)$ and the defect of global entropies $q^{-1} \sum_{j=0}^{q-1} \left( H^{n_0 - 1}_0(U^{(q-j)n_0}\Psi_{\h}) - H_0^{n_0-1}(\Psi_{\h}) \right)$.  All of the upcoming calculations are nearly identical for the ``backwards" quantities, so we will leave those to the reader.

Before proceeding, we introduce a quantity that appears in this section's main estimate:
\begin{defn}[Total-variation distance]
    Consider an $\epsilon$-logarithmic family $\{\Psi_{\h}\}_{\h}$ and let $j',n_0 \in \mathbb{N}.$  Then the {\it{total-variation distance at time $j'n_0$}} is the quantity
    \begin{equation} \label{eqn:total_var_dist}
        d^{(j'n_0)}(\h, \epsilon, \mathcal{P}_{\rm{sm}}) = \frac{1}{2}\sum_{|\bm{\beta}| = n_0} \left| \|\Pi_{\bm{\beta}} U^{j'n_0}\Psi_{\h} \|^2 - \|\Pi_{\bm{\beta}} \Psi_{\h} \|^2 \right|.
    \end{equation}
\end{defn}
It is useful to emphasize some the total-variation distance's properties.  When considering equation (\ref{eqn:F_1}) in Proposition \ref{l:approx_shift}, we have 
\begin{equation*}
    d^{(j'n_0)}(\h, \epsilon, \mathcal{P}_{\rm{sm}}) = \frac{1}{2}\sum_{|\bm{\beta}|=n_0} |F_1(\h, \bm{\beta}, \epsilon ,j'n_0)|.
\end{equation*}
Moreover, thanks to $ \sum_{|\bm{\beta}|=n_0} F_1(\h, \bm{\beta}, \epsilon, j' n_0) = \mathcal{O}_K(\h^{\infty})$,  
\begin{equation} \label{eqn:d_formula}
    d^{(j'n_0)}(\h, \epsilon, \mathcal{P}_{\rm{sm}}) = \sum_{\bm{\beta} \in I_+} F_1(\h, \bm{\beta}, \epsilon, j'n_0) + \mathcal{O}_K(\h^{\infty})
\end{equation}
where $I_+(\hbar, \epsilon, j', n_0) = \left\{ \beta \in \Sigma^{n_0} : F_1(\h, \bm{\beta}, \epsilon, j'n_0) \geq 0 \right\}$ and analogously for $I_{-}.$

\subsection{Basic lemmas}

\begin{lem}[Estimate on total-variation distance]
    Consider the hypotheses of Proposition \ref{prop:penult_bound},  set $r=0$, and consider the value of $\h_0$ returned in the conclusion.  Let $A_+ = A_+(\hbar, \epsilon, q-j, n_0, K) := \sum_{\bm{\beta} \in I_+} |\Pi_{\bm{\beta}}|^2$ where $j=1,\dots, q.$  Then for the range of parameters described in the conclusion of that proposition,
\begin{equation*}
    d^{((q-j)n_0)}(\h,\epsilon, \mathcal{P}_{\rm{sm}}) \leq  3 \|A_+\| \frac{\epsilon (q-j) n_0}{|\log \h|} 
\end{equation*}
for all $\h \leq \h_0$ (with $\h_0$ possibly smaller) and $j=1,\dots,q$. 
\end{lem}
\begin{proof}
Recall (\ref{eqn:d_formula}) and (\ref{eqn:F_1}). We use the operator formula 
\begin{align*}
U^n - e^{-in/2\h} Id = \frac{-i}{\h} \int_0^n \left(P(\h) - \frac{1}{2} \right) U^{s} e^{is/2\h} e^{-in/2\h} ds
\end{align*}
and bilinearity of the inner product to have $A_+$ appear.  Next, apply Cauchy-Schwarz and Minkowski's integral inequality to the remaining three terms.
\end{proof}

\begin{lem} \label{l:Cal_Val}
    Under the hypotheses of Proposition \ref{prop:penult_bound}, there exists an $\hbar_0$ (now possibly smaller) such that the operator $A_+$ has $\|A_+\| \leq 1 + \mathcal{O}_{K,n_0}(\h)$ for all $\h \leq \h_0$ and $j=1,\dots,q.$ 
\end{lem}
\begin{proof}
 The proof immediately follows from the semiclassical G\"arding Inequality \cite{Z12} applied in coordinate charts.  Notice that $1 - \sigma(A_+) \geq 0$, implying there exists $\hbar_0$ and $C>0$ such that $1 \geq \langle A_+ u, u \rangle - C \hbar$ for all $\hbar \leq \hbar_0$ and $u \in L^2$ normalized.
\end{proof}

In order to more precisely estimate the difference of our global entropies from the differences of our local entropies. We define the (probability) measure $\mathfrak{p}$ on $\Sigma^{n_0}$ at time $(q-j)n_0$ for the $\epsilon$-logarithmic family $\{ \Psi_{\h} \}_{\h}$ by
\begin{equation} \label{eqn:prob_measu_Sigma}
    \mathfrak{p}(\h, \bm{\beta}, \epsilon, (q-j)n_0) :=
    \begin{cases}
    \frac{F_1(\h, \bm{\beta}, \epsilon, (q-j)n_0)}{d^{((q-j)n_0)}(\h, \epsilon, \mathcal{P}_{\rm{sm}})}, & \mbox{ if } \bm{\beta} \in I_{+} \\
    0, & \mbox{ if } \bm{\beta} \in I_{-}
    \end{cases}
\end{equation} 
if $d^{((q-j)n_0)}(\h, \epsilon, \mathcal{P}_{\rm{sm}}) \neq 0$.  If $d^{((q-j)n_0)}(\h, \epsilon, \mathcal{P}_{\rm{sm}}) = 0$, we simply set $ \mathfrak{p}(\h, \bm{\beta}, \epsilon, (q-j)n_0) = 0$ for each $\bm{\beta}.$

\begin{lem}[Difference of local entropies] \label{lem:diff_loc_ent}
Consider the hypotheses of Proposition \ref{prop:penult_bound}, the value of $\h_0$ returned from the conclusion, and (\ref{eqn:prob_measu_Sigma}).  Let $N>0$ be given.

For an $\h_0$ possibly smaller, we have for all $\h \leq \h_0$,
\begin{align}
    \sum_{\bm{\beta} \in I_{-}} \eta \left( \| \Pi_{\bm{\beta}} U^{(q-j)n_0} \Psi_{\h} \|^2 \right) -  \eta \left( \| \Pi_{\bm{\beta}} \Psi_{\h} \|^2 \right) & \leq d^{((q-j)n_0)}(\h, \epsilon, \mathcal{P}_{\rm{sm}}) + \mathcal{O}_K(\h^N) \mbox{ and } \\
    \nonumber \sum_{\bm{\beta} \in I_{+}} \eta \left( \| \Pi_{\bm{\beta}} U^{(q-j)n_0} \Psi_{\h} \|^2 \right) -  \eta \left( \| \Pi_{\bm{\beta}} \Psi_{\h} \|^2 \right) &\leq d^{((q-j)n_0)}(\h, \epsilon, \mathcal{P}_{\rm{sm}}) \, \cdot \, H(\mathfrak{p}, \mathcal{P}^{\vee n_0}) \\
    & + \eta \left( d^{((q-j)n_0)}(\h, \epsilon, \mathcal{P}_{\rm{sm}})  \right) + \mathcal{O}_K(\h^N).
\end{align}
\end{lem}
\begin{proof}
    Notice for $F_1 \neq 0$ have that each $\bm{\beta}$,
    \begin{equation*}
        \frac{\eta \left( \| \Pi_{\bm{\beta}} U^{(q-j)n_0} \Psi_{\h} \|^2 \right) - \eta \left( \| \Pi_{\bm{\beta}} \Psi_{\h} \|^2 \right)}{F_1} \geq -1.
    \end{equation*}
    Thus, for $F_1 < 0$, $\eta \left( \| \Pi_{\bm{\beta}} U^{(q-j)n_0} \Psi_{\h} \|^2 \right) - \eta \left( \| \Pi_{\bm{\beta}} \Psi_{\h} \|^2 \right) \leq -F_1(\h, \bm{\beta}, \epsilon, (q- j) n_0)$ for $\beta \in I_{-}.$  Hence, $\sum_{\bm{\beta} \in I_{-}} \eta \left( \| \Pi_{\bm{\beta}} U^{(q-j)n_0} \Psi_{\h} \|^2 \right) - \eta \left( \| \Pi_{\bm{\beta}} \Psi_{\h} \|^2 \right) \leq \sum_{\bm{\beta} \in I_+} F_1(\h, \bm{\beta}, \epsilon, (q-j)n_0) + \mathcal{O}_K(\h^{N}). $  This establishes the first inequality after recalling (\ref{eqn:d_formula}).

    For the second inequality, we use that concave functions are sublinear, that is 
    \begin{equation*}
        \eta \left( \| \Pi_{\bm{\beta}} \Psi_{\h} \|^2 + F_1(\h, \bm{\beta}, \epsilon, (q-j)n_0) \right) - \eta \left( \| \Pi_{\bm{\beta}} \Psi_{\h} \|^2 \right) \leq \eta \left( F_1(\h, \bm{\beta}, \epsilon, (q-j)n_0 ) \right).
    \end{equation*}
If $d^{((q-j)n_0)}(\bullet) \neq 0$, write $F_1 = d \frac{F_1}{d}$ and expand using the properties of logarithms to arrive at the second desired inequality.  If $d^{((q-j)n_0)}(\bullet) = 0$, then (\ref{eqn:total_var_dist}) shows $F_1(\h, \bm{\beta}, \epsilon, j'n_0) = \mathcal{O}_{K}(\h^N)$ in the corresponding parameter range and therefore the inequality holds trivially.
\end{proof}

We now define the quantity that measures the defect between the (global) quantum entropies of propagated states and those of non-propagated states:
\begin{align} \label{e:defect}
D(\h, \epsilon, n_0, \mathcal{P}_{\rm{sm}}) :=  q^{-1} \sum_{j=0}^{q-1} \left[  H^{n_0 - 1}_0(U^{jn_0}\Psi_{\h}) -  H^{n_0 - 1}_0(\Psi_{\h}) \right]
\end{align}
with $\tilde{D}(\h, \epsilon, n_0, \mathcal{P}_{\rm{sm}})$ denoting the corresponding defect for the quantities $H_{-n_0}^{-1}(\Psi_{\h})$. 

\begin{cor}[Estimate on defect term] \label{cor:defect_est}
    Under the hypotheses of Proposition \ref{prop:penult_bound}, take $\h_0$ as the minimum value between the two returned by Lemmas \ref{l:Cal_Val} and \ref{lem:diff_loc_ent}.  Then, we have for all $\h \leq \h_0$
    \begin{align*}
        D(\h, \epsilon, n_0, \mathcal{P}_{\rm{sm}}) \leq & \frac{3}{2} \epsilon \cdot \frac{n_0(q+1)}{|\log \h|} + \frac{3}{2} \epsilon \cdot \frac{n_0(q+1)}{|\log \h|} \cdot \log |\mathcal{P}^{\vee {n_0}}| \\
        & + \frac{1}{q} \sum_{j=0}^{q-1} \eta \left( d^{(q-j)n_0)}(\h, \mathcal{P}_{\rm{sm}})  \right)
    \end{align*}
    where $\frac{1}{q} \sum_{j=0}^{q-1}  \eta \left( d^{(q-j)n_0)}(\h, \mathcal{P}_{\rm{sm}})  \right) < \eta(\frac{3 \epsilon}{2 \lambda_{max}})$ for $\epsilon < \frac{2 \lambda_{max}}{3 e}$.  A similar estimate holds for $\tilde{D}(\bullet)$.
\end{cor}
\begin{proof}
    Follows immediately from the previous lemmas, using that $\sum_{j=0}^{q-1} = q(q+1)/2$, and that $H(\mathfrak{p}, \mathcal{P}^{\vee n_0}) \leq \log |\mathcal{P}^{\vee {n_0}}|$.
\end{proof}
\begin{rem}
    Note that $\log |\mathcal{P}^{\vee {n_0}}| \leq \log |\mathcal{P}|^{n_0} \leq n_0 \log K$, which is unbounded in $K$.  We will derive a sharper upper bound on $|\mathcal{P}^{\vee {n_0}}|$ in the last steps of our proof in Section \ref{sect:final_steps}.
\end{rem}

For bookkeeping purposes, we restate Proposition \ref{prop:penult_bound} with our defect terms appearing:

\begin{cor} \label{c:penult_w_defect_bound}
Under the hypotheses of Proposition \ref{prop:penult_bound} and choosing the value of $\h_0$ as returned by Corollary \ref{cor:defect_est}, we have for all $\h \leq \h_0$ that the quantum entropies satisfy the following lower bound:
\begin{align} \label{eqn:finite_pressure_bound}
\nonumber & q \left( H_0^{n_0-1}(\Psi_{\h}) - \sum_{|\bm{\beta}|=n_0} 2 \mu_{\h}(\bm{\beta}) \log (\textbf{w}_{\bm{\beta}}) \right) - \sum_{|\bm{\beta}|=n_0} 2 \log (\textbf{w}_{\bm{\beta}}) \cdot\sum_{j=1}^{q-1} F_1(\h, \bm{\beta}, \epsilon, jn_0)  \\
\nonumber & q \left( H^{-1}_{-n_0} (\Psi_{\h}) - \sum_{|\bm{\beta}|=n_0} 2 \tilde{\mu}_{\h}(\bm{\beta}) \log (\textbf{v}_{\bm{\beta}}) \right) - \sum_{|\bm{\beta}|=n_0} 2 \log (\textbf{v}_{\bm{\beta}}) \cdot\sum_{j=1}^{q-1} \tilde{F}_1(\h, \bm{\beta}, \epsilon, jn_0)  \\
\nonumber & +  q D(\h,  \epsilon, n_0, \mathcal{P}_{\rm{sm}}) + q\tilde{D}(\h,  \epsilon, n_0, \mathcal{P}_{\rm{sm}})\\
\nonumber & \geq -(d-1 + c_0\delta_{Ehr}) \times |\log \h| \\
\nonumber & - \sum_{j=1}^{q-1} R_1(\h, jn_0, n_0, \epsilon) - \sum_{j=1}^{q-1} \tilde{R}_1(\h, jn_0, n_0, \epsilon)  \\
\nonumber & - \sum_{|\bm{\beta}|=r} 2 \log (\textbf{w}_{\bm{\beta}}) \cdot F_1(\h, \bm{\beta}, \epsilon, qn_0)  - \sum_{|\bm{\beta}|=r} 2 \log (\textbf{v}_{\bm{\beta}}) \cdot \tilde{F}_1(\h, \bm{\beta}, \epsilon, qn_0) \\
\nonumber & - p_0^{r-1}(U^{qn_0}\Psi_{\h},\textbf{w}) - p_{-r}^{-1}(U^{-qn_0}\Psi_{\h}, \textbf{v}) + \mathcal{O}_{K,n_0} (\h^{\gamma})\\
\end{align}
where $R_1(\bullet), \tilde{R}_1(\bullet)$ are defined in (\ref{eqn:sub_add_error}) and the new potentials $F_2(\bullet), \tilde{F}_2(\bullet)$ in (\ref{eqn:new_potentials}) are written explicitly.
\end{cor}

\subsection{Analysis of $F_1$ and corresponding potential terms}

We transition into studying the average of the errors $F_1$ coming from shift-invariance statement (see Proposition \ref{l:approx_shift}) and its influence on the new potential terms $F_2$ (\ref{eqn:new_potentials}). The following proposition follows easily from Proposition \ref{l:approx_shift}:
\begin{prop} \label{prop:sub_quant_inv}
Consider the hypotheses of Proposition \ref{prop:penult_bound}. Choose the value of $\h_0$ as returned by Corollary \ref{cor:defect_est}.  We have for all $\h \leq \h_0$:
\begin{align} \label{e:F_sc_mass}
   \nonumber & \left|  \sum_{k_0=1}^K \log(J^u(k_0)) \, \left( q^{-1} \sum_{j=0}^{q-1}  \sum_{|\bm{\beta}'|=n_0-1} F_1(\h, (\bm{\beta}',k_0), \epsilon, jn_0) \right) \right| \\
   & \leq   \Lambda \left( \frac{2 \, \epsilon}{\lambda_{\max}(\kappa_{slab})} +  \left( \ \frac{\epsilon}{\lambda_{\max}(\kappa_{slab})}\right)^{2} \right) 
\end{align}
where $\Lambda$ is as in (\ref{eqn:LambdaMax}).  A similar statement holds for $\tilde{F}_1(\bullet)$.
\end{prop}

\begin{proof}

Recall equation (\ref{eqn:F_1}) and use bilinearity.  We proceed to estimating our three terms.  Note that $\log(J^u(k_0))>0$ thanks to (\ref{eqn:Jac_bounds}) and
\begin{align*}
& \left| \sum_{|\bm{\beta}'|=n_0-1} \sum_{k_0=1}^K \log(J^u(k_0)) \langle |\Pi_{(\bm{\beta}',k_0)}|^2 \Psi_{\h}, q^{-1} \sum_{j=0}^{q-1} \Phi_{\h}(jn_0) \rangle \right| \\
& \leq \frac{\epsilon}{\lambda_{\max}(\kappa_{slab})}  \cdot \left\| ( \sum_{|\bm{\beta}'|=n_0-1} \sum_{k_0=1}^K \log(J^u(k_0)) |\Pi_{(\bm{\beta}',k_0)}|^2 ) \Psi_{\h} \right\|. 
\end{align*}
A similar estimate holds for the second term.  Note also that 
\begin{align*}
& \left| q^{-1} \sum_{j=0}^{q-1} \langle ( \sum_{|\bm{\beta}'|=n_0-1} \sum_{k_0=1}^K \log(J^u(k_0)) |\Pi_{(\bm{\beta}',k_0)}|^2 )\Phi_{\h_k}(jn_0), \Phi_{\h_k}(jn_0) \rangle \right| \\
& \leq  q^{-1} \sum_{j=0}^{q-1} \left\| ( \sum_{|\bm{\beta}'|=n_0-1} \sum_{k_0=1}^K \log(J^u(k_0)) |\Pi_{(\bm{\beta}',k_0)}|^2 ) \Phi_{\h_k}(jn_0) \right\| \frac{\epsilon}{\lambda_{\max}(\kappa_{slab})}.
\end{align*}
Note that $\| ( \sum_{|\bm{\beta}'|=n_0-1} \sum_{k_0=1}^K \log(J^u(k_0)) |\Pi_{(\bm{\beta}',k_0)}|^2 ) \|_{L^2 \circlearrowright} \leq \Lambda$.  The sequences $\{ \Phi_{\h_k}[jn_0] \}_{\h}$ yield positive measures (each of total weight $\leq (\frac{\epsilon}{\lambda_{\max}(\kappa_{slab})})^2$ thanks to (\ref{eqn:F_1})).  
\end{proof}



Let us consider $ \sum_{j} F_2(\bullet)$ in (\ref{eqn:finite_pressure_bound}), multiply by $q^{-1}$, and proceed as follows: 
\begin{align} \label{eqn:new_potentials_simplified}
 \nonumber & -  \sum_{|\bm{\beta}|=n_0} 2 \,  \log (w_{\bm{\beta}}) \left( q^{-1} \sum_{j=0}^{q-1} F_1(\h,\bm{\beta}, \epsilon, jn_0) \right) \\
\nonumber  & =  \sum_{ |\bm{\beta}'|=n_0-1 } \log(J^u_{n_0-1}(\bm{\beta}'))  \, \left( q^{-1} \sum_{j=0}^{q-1} \left( \sum_{k_0=1}^K F_1(\h, (\bm{\beta}',k_0), \epsilon, jn_0) \right) \right) \\
 & + \sum_{k_0=1}^K \log(J^u(k_0)) \, \left( q^{-1} \sum_{j=0}^{q-1}  \sum_{|\bm{\beta}'|=n_0-1} F_1(\h, (\bm{\beta}',k_0), \epsilon, jn_0) \right),
\end{align} 
with the second sum appearing in the statement of Proposition \ref{prop:sub_quant_inv}.  A similar expression holds for the average of $\tilde{F}_2(\bullet).$  This expression, and the proposition connected to it, will be used again in Section \ref{sect:final_steps}.

\section{Parameter selections and reductions to classical entropy} \label{sect:final_steps}

We now execute the final steps of the proof of Theorem \ref{thm:main}.  Recall the parameters listed in Section \ref{sect:param_list} and the notion of entropy for smooth partitions as described in Appendix \ref{sect:dynam_defs}. We proceed exactly as in \cite[Section 2.2.8]{AN07} and \cite[Section 3.10]{Non12}, but continue to provide all the details for readability.

\subsection{First lower bound on $H^{n_0-1}_0$}

Recall Corollary \ref{c:penult_w_defect_bound}.  Let $qn_0 = \lfloor 2 T_{\delta_{Ehr}, \kappa_{slab}, \h} \rfloor - r = \frac{1 + \mathcal{O}(\delta + \frac{r \lambda_{max}}{|\log \h|})}{\lambda_{max}(\kappa_{slab})} |\log \h|$.   Bounding the sum of our new potential terms $F_2(\bullet), \tilde{F}_2(\bullet)$ by their absolute values and subtracting, we have the following inequality for all $\h \leq \h_0$ for $\h_0$ possibly smaller:

\begin{align} \label{eqn:pre_h_lb}
\nonumber   \frac{H_0^{n_0-1}(\Psi_{\h})}{2} + & \frac{H_{-n_0}^{-1}(\Psi_{\h})}{2}  +  \frac{D(\h, n_0, \epsilon, \mathcal{P}_{\rm{sm}})}{2} +  \frac{\tilde{D}(\h, n_0, \epsilon, \mathcal{P}_{\rm{sm}})}{2} \\
\nonumber & \geq \frac{-(d-1 + c_0 \delta_{Ehr})) \times |\log \h|}{2q}  \\
\nonumber & + \frac{1}{2} \left( \sum_{|\bm{\beta}|=n_0} 2 \mu_{\h}(\bm{\beta}) \log (w_{\bm{\beta}}) + \sum_{|\bm{\beta}|=n_0} 2 \tilde{\mu}_{\h}(\bm{\beta}) \log (w_{\bm{\beta}}) \right) \\
\nonumber & - \frac{1}{2} \left| \sum_{|\bm{\beta}|=n_0} 2 \log (w_{\bm{\beta}}) \cdot \left( q^{-1}  \sum_{j=1}^{q-1} F_1(\h, \bm{\beta}, jn_0, n_0) + q^{-1}  \sum_{j=1}^{q-1} \tilde{F}_1(\h, \bm{\beta}, jn_0, n_0)\right) \right| \\
& +  \mathcal{O}_{K,n_0}(\h^{\gamma/2}) + \mathcal{O}_r(|\log \h|^{-1}).
\end{align}
Set $\kappa_{slab}=\delta_{Ehr} $.

\subsection{Manipulating the potential terms}
 First, recall (\ref{eqn:new_potentials_simplified}).  We repeat this splitting of sums $n_0-1$ more times  in (\ref{eqn:new_potentials_simplified}), and apply the procedure in the proof of Proposition \ref{prop:sub_quant_inv} to each of the $n_0$ terms, to obtain
\begin{equation} \label{eqn:new_potential_LB}
\left| \sum_{|\bm{\beta}|=n_0} 2 \log (\textbf{w}_{\bm{\beta}}) \cdot \left( q^{-1}  \sum_{j=1}^{q-1} F_1(\h, \bm{\beta}, jn_0, n_0) \right) \right| \leq n_0 \cdot \Lambda \cdot \left(  \frac{2 \epsilon}{\lambda_{\max}(\delta_{Ehr})} +   \left( \ \frac{\epsilon}{\lambda_{\max}(\delta_{Ehr})}\right)^{2}  \right)
\end{equation}
and similarly for the backwards quantities.  Apply this to the new potential terms in the left-hand side of (\ref{eqn:pre_h_lb}) to obtain a new lower bound.  Let $\h \rightarrow 0$.

 Recall that $\mu_{sc}$ is $g^t$-invariant, allowing us to equate the (original) potential and entropy terms arising from the forward and backwards quantities.  We arrive at the following simplified inequality after applying Corollary \ref{cor:defect_est} to the entropy defect terms $D(\bullet), \tilde{D}(\bullet)$:
\begin{align} \label{eqn:prelimit_ineq}
\nonumber   H_0^{n_0-1}(\mu_{sc}, \mathcal{P}_{\rm{sm}})  &  \\
\nonumber & \geq \frac{-(d-1 + c\delta) \times (\lambda_{\max} + \mathcal{O}(\delta_{Ehr}))}{2} n_0  \\
\nonumber & + \sum_{|\bm{\beta}|=n_0} 2 \mu_{sc}(\pi_{\bm{\beta}}) \log (\textbf{w}_{\bm{\beta}})   \\
\nonumber &  -   n_0 \cdot \Lambda \cdot \left(  \frac{2 \epsilon}{\lambda_{\max}(\delta_{Ehr})} +   \left( \ \frac{\epsilon}{\lambda_{\max}(\delta_{Ehr})}\right)^{2}  \right)\\
 & - \frac{3}{2} \times  \epsilon \times \frac{1 + \mathcal{O}(\delta_{Ehr})}{\lambda_{\max} + \mathcal{O}(\delta_{Ehr})} - \frac{3}{2} \times  \epsilon \times \frac{1 + \mathcal{O}(\delta_{Ehr})}{\lambda_{\max} + \mathcal{O}(\delta_{Ehr})} \times \log |\mathcal{P}^{\vee n_0}| - e^{-1}.
\end{align}  

Suppose we were able to replace each instance of $\pi_{\bm{\beta}}$ with the indicator function on $E_{\bm{\beta}}$ in the potential $\sum_{|\bm{\beta}|=n_0} 2 \mu_{sc}(\pi_{\bm{\beta}}) \log (\textbf{w}_{\bm{\beta}})$, after letting $\h \rightarrow 0$.  Note that we have the following convenient simplification (which itself inspired the simplification in (\ref{eqn:new_potentials_simplified}) and hence Proposition \ref{prop:sub_quant_inv}):
\begin{align} \label{eqn:old_potential_simplified}
 \sum_{|\bm{\beta}|=n_0} \log(J_{n_0}^u(\bm{\beta})) \times \mu_{sc}(E_{\bm{\beta}} )
 \nonumber & = \sum_{k=1}^K \sum_{\{ | \bm{\beta}| = n_0, \beta_0=k \}} \left( \sum_{l=1}^{n_0-1} \log(J^u(\beta_{j}) + \log (J^u(k)) \right) \mu_{sc}(E_{\bm{\beta}}) \\
 \nonumber & = \sum_{k=1}^K \sum_{\{ | \bm{\beta}'| = n_0-1 \}} \left( \sum_{l=1}^{n_0-1} \log(J^u(\beta_{j})\right) \mu_{sc}(E_{(\bm{\beta}, k)}) + \sum_{k=1}^K \log (J^u(k)) \mu_{sc}(E_{k}) \\
\nonumber  & = \sum_{|\bm{\beta}'| = n_0 -1} \underbrace{\left( \sum_{j=1}^{n_0-1} \log \left( J^u(\beta_j) \right) \right)}_{\log(J^u_{n_0-1}(\bm{\beta}'))} \mu_{sc}(E_{\bm{\beta}}) + \sum_{k=1}^K \log (J^u(k)) \mu_{sc}(E_{k})\\
 \nonumber & \vdots \\
 & =  n_0 \sum_{E_{k} \in \mathcal{P}} \mu_{sc} \left( E_k \right) \log(J^u(k)).
\end{align}
In the last section of this note, we will show how to replace the smoothed cutoffs appearing in these potential terms in (\ref{eqn:penult_prelimit_ineq}).

\subsection{The term $\log |\mathcal{P}^{\vee n_0}|$}
The last difficulty is in replacing $\log |\mathcal{P}^{\vee n_0}|$ by another quantity so that we do not recover a trivial lower bound when taking $K \rightarrow \infty$.  To resolve this, we re-do all of our analysis but for the geodesic flow at time step $n_1$ (that is, $g^{n_1}$ on $\mathcal{E}$ or equivalently $g^1$ on $\{ \|\xi\| = n_1 \}$) and the re-scaled Laplacian $P(\h(n_1)) := - (\h n_1)^2 \frac{\Delta_g}{2}$ at energy $\frac{n_1^2}{2}$ with our new semiclassical parameter being $\h(n_1) := n_1 \h$.  Define $\lambda_{\max}(n_1)(\delta_{Ehr})$ to be the maximal expansion rate on the energy slab $\mathcal{E}_{ \delta_{Ehr}}(n_1) = \{(x\;\xi) \in T^*M \, : \, |\xi| \in (n_1 -  \delta_{Ehr}, n_1 + \delta_{Ehr}) \}.$ 

Notice that for our original $\epsilon$-logarithmic family $\{ \Psi_{\h} \}_{\h}$ for $P(\h)$ with energy $\frac{1}{2}$, $\Psi_{h}$ satisfies 
\begin{align*}
    \left\| \left( P(\h(n_1)) - \frac{n_1^2}{2} \right) \Psi_{\h} \right\| \leq n_1 \epsilon \frac{\h(n_1)}{| \log \left( \frac{\h(n_1)}{n_1} \right) |} = o(\h(n_1))
\end{align*}
and therefore yielding an $g^{n_1}$-invariant semiclassical measure whose entropy $H_0^{n}(\Psi_{\h}, \mathcal{P}_{\rm{sm}}(n_1))$ we can analyze for $n \leq \lfloor 2 \frac{(1-\delta_{Ehr})}{2 \lambda_{max}(n_1)(\delta_{Ehr})} |\log \h(n_1)| \rfloor$ where the initial partition is $\mathcal{P}(n_1) := \mathcal{P}^{\vee n_1}$.  
 Furthermore, via homogeneity properties, the quantities $\lambda_{\max}(n_1)(\delta_{Ehr})$ and $\Lambda(n_1) := \sup_{\rho \in \{ \|\xi\| = n_1\}} \log (J^u(\rho)) = \sup_{\rho \in \{ \|\xi\| = 1\}} \log (J^u(\rho)^{n_1})$ satisfy 
 \begin{equation} \label{eqn:scaling_dynam_quants}
 \lambda_{\max}(n_1)(\delta_{Ehr}) = n_1 \lambda_{\max}(\delta_{Ehr}/n_1), \, \, \, \Lambda(n_1) = n_1 \Lambda
 \end{equation}
 and therefore
\begin{equation} \label{eqn:homog_0}
    \frac{n_1 \epsilon}{\lambda_{max}(n_1)(\delta_{Ehr})} =\frac{\epsilon}{\lambda_{max}(\delta_{Ehr}/n_1)}.
\end{equation}
 
 We leave it to the reader to verify that all of the analysis in the previous sections carries through almost verbatim with respect to the parameter $\h(n_1)$. Set $\{ \Psi(\h(n_1)) \}_{\h(n_1)} = \{ \Psi_{\h} \}_{\h}$.  Take $\h(n_1) \rightarrow 0$ and for $\mu_{sc}(n_1)$ the corresponding semiclassical measure, we remind ourselves that $\mu_{sc} = \mu_{sc}(n_1)$. 

 Recall the notion of topological entropy $H_{top}$ used in \cite[Page 445]{An08} and apply it to the diffeomorphism $g^{1}$.  Let $\gamma' > 0$ be given.  For Anosov flows, via the use of the coding map, there exists $C'(K)>1$ such that of the $K^{n_1}$ possible elements in $\mathcal{P}(n_1)$, we have at most $C' e^{n_1(H_{top} + \gamma')}$ being non-empty for $n_1 \geq N_1(K,\gamma')$.  In fact, this gives us a refined bound of 
 \begin{equation} \label{eqn:top_ent_ub}
     |\mathcal{P}(n_1)^{\vee n_0}| \leq C'' e^{n_1  n_0(H_{top} + \gamma')}
 \end{equation}
 for $N_1$ possibly larger and $n_0 \geq N_0(K,N_1,\gamma')$, where $C''(K)>0$.

\subsection{The final steps}
Consider the discussion after Definition \ref{defn:smooth_entropy_pressure}.  Recall the property that our initial partition $\mathcal{P}$'s elements $E_k$ have boundaries that are not charged by $\mu_{sc}$.  A convolution argument involving $\mathcal{P}_{\rm{sm}}$ (use the approximations-to-the-identity parameter $\eta$), and the pressure functional being continuous in this limit, allows us to replace all instances of smooth cutoffs $\pi_{\bm{\beta}}$ in $H_{0}^{n_0-1}(\mu_{sc}, \mathcal{P}_{\rm{sm}})$ and the (old) potentials corresponding to $\mu_{sc}$ with the sharp indicator functions on $E_{\bm{\beta}} \subset \mathcal{E}_{\delta_{Ehr}}$. 

Take $\eta \rightarrow 0$ in (\ref{eqn:prelimit_ineq}).  Now use (\ref{eqn:old_potential_simplified}), the identities (\ref{eqn:scaling_dynam_quants}) and (\ref{eqn:homog_0}), and the upper bound (\ref{eqn:top_ent_ub}).  Thus, we obtain a penultimate lower bound for the entropy (of our original semiclassical measure)
\begin{align} \label{eqn:penult_prelimit_ineq}
\nonumber   H_0^{n_0-1}(\mu_{sc}, \mathcal{P}(n_1))  &  \\
\nonumber  \geq n_0 n_1 & \Bigg[ \frac{-(d-1 + c  \delta_{Ehr}) \times (\lambda_{\max} + \mathcal{O}( \delta_{Ehr}/n_1))}{2}   \\
\nonumber & +  \sum_{E_{k} \in \mathcal{P}} \mu_{sc} \left( E_k \right) \log(J^u(k))  \\
\nonumber &  -  \Lambda\cdot \left( \frac{2 \epsilon}{\lambda_{\max} + \mathcal{O}((\delta_{Ehr}/n_1))} +  \left( \ \frac{\epsilon}{\lambda_{\max} + \mathcal{O}((\delta_{Ehr}/n_1))}\right)^{2} \right) \\
 \nonumber & - \frac{1}{n_0 n_1} \left(\frac{3}{2} \times  \frac{  \epsilon(1 +  \mathcal{O}( \delta_{Ehr}))}{\lambda_{\max} + \mathcal{O}((\delta_{Ehr}/n_1))} \right) \\
 \nonumber & - \frac{3}{2} \times  \frac{ \epsilon(1 + \mathcal{O}(\delta_{Ehr}))}{\lambda_{\max} + \mathcal{O}((\delta_{Ehr}/n_1))} \times (H_{top} + \gamma') \\
 & - \frac{1}{n_1 n_0} \left( \frac{3}{2} \times  \frac{ \epsilon(1 + \mathcal{O}(\delta_{Ehr}))}{\lambda_{\max} + \mathcal{O}((\delta_{Ehr}/n_1))} \times \log C'' - \frac{1}{n_0n_1 e}  \right) \Bigg]
\end{align}

Finally, we multiply both sides of (\ref{eqn:penult_prelimit_ineq}) by $(n_1 n_0)^{-1}$ and take the remaining steps in the following order: $\delta_{Ehr} \rightarrow 0$, $n_0\rightarrow \infty$, and use that both $\mathcal{P}(n_1)$ is generating and the KS entropy scales by the time step $n_1$ giving us $\frac{1}{n_0n_1} H_0^{n_0-1}(\mu_{sc}, \mathcal{P}(n_1)) \rightarrow H_{KS}(\mu_{sc})$. Next, take $\gamma' \rightarrow 0$ (therefore making $n_1 \rightarrow \infty$) and $K \rightarrow \infty$ to complete our proof of Theorem \ref{thm:main}.

\newpage

\appendix

\section{On the semiclassical calculus} \label{sect:semiclassics}

In this appendix we recall the concepts and definitions from semiclassical analysis needed in our work. The notations are drawn from the monographs \cite{DimSj99,Z12} as well as \cite{AN07}.

\subsection{Semiclassical calculus on $M$}

Recall that we define on $\mathbb{R}^{2d}$ the following class of symbols for $m \in \R$:
\begin{eqnarray} \label{symbolclass}
S^{m}(\mathbb{R}^{2d}) :=\{a \in C^{\infty}(\mathbb{R}^{2d} \times (0, 1] ): |\partial^{\alpha}_x\partial^{\beta}_{\xi} a| \leq C_{\alpha,\beta} \langle\xi\rangle^{m-|\beta|} \}.
\end{eqnarray}
Symbols in this class can be quantized through the $\h$-Weyl quantization into the following pseudodifferential operators acting on $u\in C^{\infty}(\mathbb{R}^d)$:
\begin{equation*}
\Oph^{w}(a)\,u(x) :=
\frac{1}{(2\pi\hbar)^d}\int_{\mathbb{R}^{2d}}e^{\frac{i}{\hbar}\langle x-y,\xi\rangle}\,a\big(\frac{x+y}{2},\xi;\hbar\big)\,u(y)dyd\xi\,.
\end{equation*}
One can adapt this quantization procedure to the case of the phase space $T^*M$, where $M$ is a smooth compact manifold of dimension $d$ (without boundary). Consider a smooth atlas $(f_l,V_l)_{l=1,\ldots,L}$ of $M$, where each $f_l$ is a smooth diffeomorphism from $V_l\subset M$ to a bounded open set $W_l\subset \R^{d}$. To each $f_l$ corresponds a pullback $f_l^*:C^{\infty}(W_l)\rightarrow C^{\infty}(V_l)$ and a symplectic diffeomorphism $\tilde{f}_l$ from $T^*V_l$ to $T^*W_l$: 
$$
\tilde{f}_l:(x,\xi)\mapsto\left(f_l(x),(Df_l(x)^{-1})^T\xi\right).
$$
Consider now a smooth partition of unity $(\phi_l)$ adapted to the previous atlas $(f_l,V_l)$. 
That means $\sum_l\phi_l=1$ and $\phi_l\in C^{\infty}(V_l)$. Then, any observable $a$ in $C^{\infty}(T^*M)$ can be decomposed as: $a=\sum_l a_l$, where 
$a_l=a\phi_l$. Each $a_l$ belongs to $C^{\infty}(T^*V_l)$ and can be pushed to a function $\tilde{a}_l=(\tilde{f}_l^{-1})^*a_l\in C^{\infty}(T^*W_l)$.
We may now define the class of symbols of order $m$ on $T^*M$ (after slightly abusing notation and treating $(x, \xi)$ as coordinates on $T^*W_l$)
\begin{align*} \label{pdodef}  
S^{m}(T^*M)   := \{a \in C^{\infty}(T^*M \times (0, 1] ): a=\sum_l a_{l},\  \text{ such that } \tilde{a}_l\in S^m(\R^{2d})\quad\text{for each }l\}.
\end{align*}
This class is independent of the choice of atlas or smooth partition. 
For any $a\in S^{m}(T^{*}M)$, one can associate to each component $\tilde{a}_l\in S^{m}(\mathbb{R}^{2d})$ its Weyl quantization $\Oph^w(\tilde{a}_l)$, which acts on functions on $\R^{2d}$.
To get back to operators acting on $M$, we consider smooth cutoffs $\psi_l\in C_c^{\infty}(V_l)$ such that $\psi_l=1$ close to the support of $\phi_l$, and define the operator:
\begin{equation*}
\label{pdomanifold}\Oph(a)u :=
\sum_l \psi_l\times\left(f_l^*\Op_{\hbar}^w(\tilde{a}_l)(f_l^{-1})^*\right)\left(\psi_l\times u\right),\quad u\in C^{\infty}(M)\,.
\end{equation*}
This quantization procedure maps (modulo smoothing operators with seminorms $\mathcal{O}(\hbar^{\infty})$) symbols $a\in S^{m}(T^{*}M)$ onto the space $\Psi^{m}_\hbar(M)$ of semiclassical pseudodifferential 
operators of order $m$. The dependence in the cutoffs $\phi_l$ and $\psi_l$ only appears at order 
$\hbar\Psi^{m-1}_\hbar$ \cite[Thm 9.10]{Z12}), so that the principal symbol map $\sigma_0:\Psi^{m}_\hbar(M)\rightarrow S^{m}(T^*M)/\h S^{m-1}(T^*M)$ is 
intrinsically defined. Most of the rules and microlocal properties (for example the composition of operators, the Egorov and Calder\'on-Vaillancourt Theorems) that hold on $\mathbb{R}^{d}$ can be extended to the manifold case.

An important example of a pseudodifferential operator is the semiclassical Laplace-Beltrami operator $P(\h)=-\h^2 \Delta_g$.  In local coordinates $(x; \xi)$ on $T^*M$, the operator can be expressed as $\Op_h^w \big( |\xi|^2_g + \h (\sum_j b_j(x) \xi_j + c(x)) + \h^2 d(x) \big)$ for some functions $b_j,c,d$ on $M$.  In particular, its semiclassical principal symbol is the function $|\xi |^2_g \in S^{2}(T^*M)$. 

We will need to consider a slightly more general class of symbols than those in \eqref{symbolclass}.  Following \cite{DimSj99}, for any $0 \leq \nu < 1/2$ we introduce the symbol class
\begin{eqnarray} \label{singsymbolclass}
\nonumber S^{m,k}_{\nu}(\mathbb{R}^{2d}) := \{a \in C^{\infty}(\mathbb{R}^{2d} \times (0, 1] ): |\partial^{\alpha}_x\partial^{\beta}_{\xi} a| \leq C_{\alpha,\beta} \h^{-k-\nu|\alpha + \beta|}\langle\xi\rangle^{m-|\beta|} \}.
\end{eqnarray}
These symbols are allowed to oscillate more strongly when $\h\to 0$.
All the previous remarks regarding the case of $\nu=0$ transfer over in a straightforward manner. This slightly ``exotic'' class of symbols can be adapted on $T^*M$ as well. For more details, see \cite[Section 14.2]{Z12}.

\subsection{Anisotropic calculus and sharp energy cutoffs} \label{sect:aniso_sharp_cutoffs}

This section is a summary of \cite[Section 2.3]{Non12} that addresses the need to quantize observables which are ``very singular" along certain directions, away from some specific submanifold, \`a la the calculus introduced by Sj\"ostrand-Zworski \cite{SjZw99} (see \cite[Section 5.3]{AN07} and also \cite{DJN19} for further results on singular symbol calculi).  The material here is used when working with cutoffs in thin neighbourhoods of $\mathcal{E}$, that is, when establishing Proposition \ref{prop:pou_quasi} and Proposition \ref{prop:hyp_disp}.  It is also important in Section \ref{sect:micro_props}.

Let $\Sigma \subset T^*M$ be a compact co-isotropic manifold of dimension $2d-D$ (where $D \leq d$) and $\kappa_{slab}>0$ be fixed.  Near each point $\rho \in \Sigma$, there exists local canonical coordinates $(y_i, \eta_i)$ such that 
\begin{equation*}
\Sigma = \{ \eta_1 = \eta_2 = \dots = \eta_D = 0\}.
\end{equation*}
For some index $\nu \in [0,1)$, we define as follows a class of smooth symbols $a \in S^{m,k}_{\Sigma, \nu}(T^*M) \subset C^{\infty}(T^*M \times (0,1]):$
\begin{itemize}
\item for any family of smooth vector fields $V_1, \dots, V_{l_1}$ tangent to $\Sigma$ and of smooth vector fields $W_1, \dots, W_{l_2}$, we have in any tubular neighbourhood $\Sigma_{\kappa_{slab}}$ of $\Sigma$:
\begin{equation*}
\sup_{\rho \in \Sigma_{\kappa_{slab}}} \left| V_1 \dots V_{l_1} W_1 \dots W_{l_2} a(\rho) \right| \leq C \hbar^{-k - \nu l_{2}},
\end{equation*}
\item away from $\Sigma$, we require that $|\partial_x^{\alpha} \partial_{\xi}^{\beta} a |  = \mathcal{O}( \hbar^{-k} \langle \xi \rangle^{m - |\beta|})$.
\end{itemize}
Loosely speaking, these symbols $a$ can be split into components each supported in some adapted coordinate chart mapping to $\Sigma$ per standard procedures, with one piece $a_{\infty} \in S^{m,k}(T^*M)$ vanishing near $\Sigma$.  The Weyl quantization along with the use of zeroth-order Fourier integral operators associated to canonical graphs allows for a global quantization procedure which we call $\Op_{\Sigma, \h}$ and whose corresponding class we name $\Psi_{\Sigma,\nu}^{m,k}(T^*M)$.  The exact steps are given in \cite[Section 5.3.2]{AN07}.  We now set $\Sigma = \mathcal{E} = \{|\xi|^2 = 1/2\}$.   

To set up the relevant pseudors, we first introduce some important cutoffs.  Let  $\delta_{energy}>0$ be given and $n \leq C_{\delta_{energy}} |\log \h|$ where $C_{\delta_{energy}} + 1 < \delta^{-1}_{energy}$.  Define $\chi_{\delta_{energy}}(s)$ equal to 1 for $|s| \leq e^{-\delta_{energy}/2}$ and 0 for $|s| \geq 1$.  From here we set
\begin{equation*}
\chi^{(n)}(s,\h) := \chi_{\delta_{energy}}(e^{-n\delta_{energy}} h^{-1+\delta_{energy}} s).
\end{equation*}
This cutoff has width $2\h^{1-(1+C_{\delta_{energy}})\delta_{energy}} \sim \h^{\gamma'}$ for some $\gamma'>0$.  The symbols $\chi^{(n)}(|\xi|^2 - 1/2) \in S_{\mathcal{E}, 1 - \delta_{energy}}^{-\infty,0}$ and they will be quantized into $\Op_{\mathcal{E}, \h}(\chi^{(n)}(|\xi|^2 - 1/2))$.  The following lemma is constantly used in the background especially when working with the quasiprojectors from Section \ref{sect:quasiprojs}:

\begin{lem}[Disjoint supports] (cf \cite[Lemma 2.5]{Non12}) \label{lem:disj_supp}
Let $\delta_{energy}>0$ be given.  For any symbol $a \in S^{m,0}_{\mathcal{E}, 1 - \delta_{energy}}$ and any $0 \leq n \leq C_{\delta_{energy}} |\log \h|$, one has
\begin{equation*}
\left( Id - \Op_{\mathcal{E}, \h} \left( \chi^{(n+1)} \circ (|\xi|^2 -1/2) \right) \right) \, \Op_{\mathcal{E}, \h} \left( a \right) \, \Op_{\mathcal{E}, \h} \left( \chi^{(n)} \circ (|\xi|^2 - 1/2)\right) = \mathcal{O}(\h^{\infty}).
\end{equation*}
The same property holds if we replace $\Op_{\mathcal{E}, \h} \left( \chi^{(n)} \circ (|\xi|^2 - 1/2)\right)$ by $\chi^{(n)}(P(\h) - 1/2),$ the latter defined through spectral calculus.
\end{lem}

\section{Entropy, Pressure, and Quasiprojectors} \label{sect:entropy_quasiproj}

The material in this appendix is a summary of \cite[Section 3]{Non12}.  For more information on entropy for dynamical systems, see the classical texts of \cite{HK} and \cite[Chapter 9]{BS}.  Throughout this section, $g^t: \mathcal{E} \rightarrow \mathcal{E}$ denotes the geodesic flow on the energy shell $\mathcal{E}$.

\subsection{Classical entropy and pressure} \label{sect:dynam_defs}

\subsubsection{Classical entropy}

\begin{defn} \label{defn:classical_entropy}
Let $\mu$ be a $g^t$-invariant probability measure on $\mathcal{E}$ and $\mathcal{P} = \{E_1, \dots, E_K\}$ be a $\mu$-measurable partition on $\mathcal{E}$.  We say the  \textit{(metric) entropy} of $\mu$ with respect to $\mathcal{P}$ is the quantity 
\begin{equation*}
H(\mu, \mathcal{P}) := \sum_{k=1}^K \eta(\mu \left(E_k \right) ) \mbox{ where } \eta(s) := - s \log s \mbox{ and } s \in [0,1]  
\end{equation*}
\end{defn}
\noindent and call $\eta$ the \textit{local (metric) entropy at $s$}. 

It is important to emphasize that in our article, we first take an open cover $\{ \mathscr{O}_k'\}_{k}$ of $M$ with $M = \sqcup_{k=1}^K E'_k$, a partition consisting of Borel sets $E_k' \subseteq \mathscr{O}_{k}'$ with non-empty interiors.  Next, we take the elements $E_k := \pi_{\mathcal{E}}^{-1}(E'_k)$ where $\pi_{\mathcal{E}}: \mathcal{E} \rightarrow M$ is the canonical projection. 

In order to compute what is referred to as the \textit{Kolmogorov-Sinai entropy} of $\mu$, we first need to consider dynamical refinements $\mathcal{P}^{\vee n}$ (or alternatively written as $ [ \mathcal{P} ]_0^{n-1}$ in \cite{Non12}) where $n \geq 1$ and whose elements take the form
\begin{equation*}
E_{\bm{\alpha}} := g^{-(n-1)} E_{\alpha_{n-1}} \cap \dots g^{-1} E_{\alpha_1} \cap E_{\alpha_0}
\end{equation*}
where 
\begin{align*}
\bm{\alpha} := \alpha_0 \dots \alpha_{n-1} \in \Sigma^{n} := \{1, \dots, K \}^n
\end{align*}
is a \textit{word of length $n$ with symbols $\alpha_k \in \{1, \dots, K \}$}. Note that many such $E_{\bm{\alpha}}$ will be empty however we continue to sum over all possible words of length $n$ when computing the metric entropies.  We also write $\mathcal{P}^{\vee -n}$ to be the refinement with elements
\begin{equation*}
g^{n} E_{\bm{\alpha}} := g^{n} E_{\alpha_{-n}} \cap \dots \cap g^{1} E_{\alpha_{-1}}. 
\end{equation*}
The elements $E_{\bm{\alpha}}$ consist of points $\rho$ in $E_{\alpha_0}$ such that $g^{j}(\rho) \in E_{\alpha_j}$ for $j \leq n-1$ with a similar description for $g^nE_{\bm{\alpha}}$. With all this notation in place, if the initial partition $\mathcal{P}$ is unambiguous, we simply write
\begin{equation*}
H_{0}^{ n - 1}(\mu) := H \left( \mu, \mathcal{P}^{\vee n} \right).
\end{equation*}

\begin{defn}
Given a $g^t$-invariant probability measure $\mu$, the \textit{Kolmogorov-Sinai entropy $H_{KS}(\mu)$} is the number
\begin{equation*}
H_{\rm{KS}}(\mu) := \sup_{\mathcal{P}} \lim_{n \rightarrow \infty} \frac{1}{n} H_0^{n-1}(\mu) < \infty.
\end{equation*}
\end{defn}
 Note that there exists $\zeta_g > 0$ such that if $diam \{  \mathscr{O}_k' \}_k \leq \zeta_g$, then $\mathcal{P}^{\vee 2}$ is generating \cite[Page 7]{AKN06}, thanks to the theory of expansive flows. 
 If a partition $\mathcal{P}$ is generating, this supremum is obtained \cite[Page 7]{AKN06}.  We use this fact in Section \ref{sect:final_steps} when taking $\zeta_g \rightarrow 0$.

\subsubsection{Classical pressure}

Given our $g^t$-invariant probability measure $\mu$ on $\mathcal{E}$ and a partition $\mathcal{P} = \{E_1, \dots, E_K\}$, we can associate a set of weights 
\begin{equation} \label{eqn:pressure_weights}
\bm{w}=\{w_i >0 : k=1,\dots, K \}
\end{equation}
to this pairing.  This leads to the notion of pressure.
\begin{defn}
We define the \textit{pressure} as 
\begin{equation*}
p( \mu, \mathcal{P}, \bm{w}) := - \sum_{k=1}^K \mu(E_k) \log \left( w_k^2 \, \mu(E_k) \right).
\end{equation*}
\end{defn}
\noindent We will refer to the sum $- \sum_{k=1}^K \mu(E_k) \log \left( w_k^2\right)$ as the \textit{potential} throughout our article.

This naturally leads to the question of how to define our weights when the partition $\mathcal{P}$ is refined via the dynamics of $g^t$.  Thus, we introduce the next definition:
\begin{defn}
Given a set of weights $\bm{w}=\{w_i >0 : k=1,\dots, K \}$ associated to a partition $\mathcal{P}$, we define the \textit{n-th refinement} of $\bm{w}$ as the set consisting of the elements
\begin{equation*}
w_{\bm{\alpha}} = \prod_{j=0}^{n-1} w_{\alpha_j}, \, \mbox{ with } \, |\bm{\alpha}|=n,
\end{equation*}
where $w_{\bm{\alpha}}$ is attached to the element $E_{\bm{\alpha}} \in \mathcal{P}^{\vee n}$.  Thus, we write the pressure associated to the partition $\mathcal{P}^{\vee n}$ and these refined weights as 
\begin{align*}
p_0^{n-1}(\mu, \mathcal{P}, \bm{w}).
\end{align*}
\end{defn}

\subsubsection{Entropy of smooth partitions} \label{sect:entropy_smooth}

First let $\mathcal{P} = \{E_1, \dots, E_K\}$ be the $\mu$-measurable partition of $\mathcal{E}$ discussed directly after Definition \ref{defn:classical_entropy}.  Let $\mathbf{1}_k$ be the characteristic function on $E_k$ and define
\begin{equation} \label{eqn:char_func_elem}
\mathbf{1}_{\bm{\alpha}} : = \mathbf{1}_{\alpha_{n-1}} \circ g^{n-1} 
\times \dots \times \mathbf{1}_{\alpha_{1}} \circ g^{1} \times \mathbf{1}_{\alpha_{0}} 
\end{equation}
which is the characteristic function of $E_{\bm{\alpha}}$.  Note $\mu(E_{\bm{\alpha}}) = \int_{\mathcal{E}} \mathbf{1}_{\bm{\alpha}}  \, d \mu$.  This leads us to our next definition:
\begin{defn} \label{def:smooth_partition}
Given any $\kappa_{slab}>0$,  define 
\begin{align*}
\mathcal{E}_{\kappa_{slab}} = \{ (x,\xi) \in T^*M \, | \, \frac{1}{2} - \kappa_{slab} \leq \| \xi \|_{g(x)} \leq \frac{1}{2} + \kappa_{slab} \}.
\end{align*} 
There exists a family of functions $\{ \pi_k \}_{k=1}^K$ and neighbourhoods $\{\mathscr{O}_k\}_{k}$ in $\mathcal{E}_{\kappa_{slab}}$ with the following properties: $\pi_k \in C^{\infty}_0(\mathcal{E}_{\kappa_{slab}}, [0,1])$, $\supp \pi_k \subseteq \mathscr{O}_k$, $E_k \subseteq \mathscr{O}_k$ with $\pi_{k \upharpoonright E_k} = 1$, and
\begin{equation*} 
\sum_{k=1}^K \pi_k := \chi_{\kappa_{slab}} \mbox{ with } \, \supp \chi_{\kappa_{slab}} \subset \mathcal{E}_{\kappa_{slab}} \, \mbox{ and } \chi_{\kappa_{slab}} = 1 \mbox{ on } \mathcal{E}_{\kappa_{slab}/2}.
\end{equation*}
We call such a family $\mathcal{P}_{\rm{sm}} := \{ \pi_k \}_{k=1}^K$ a \textit{smooth partition of unity} near $\mathcal{E}$.  

In fact, given our choice of $\mathcal{P}$, one can set $\pi_{k} = f_{\eta} * \mathbf{1}_{E_k}$ where $f_{\eta}$ is a family of smooth, compactly supported, approximation-to-the-identities in $\eta \in (0,1)$ where $\lim_{\eta \rightarrow 0}  f_{\eta} * \mathbf{1}_{E_k} = \mathbf{1}_{E_k} $.  Set $\mathscr{O}_k := \pi_{T^*M}^{-1}(\mathscr{O}_k') \cap \mathcal{E}_{\kappa_{slab}}^{int}$ (where $ \pi_{T^*M}^{-1}$ is now the inverse of the canonical projection map from $T^*M$) and take $\eta$ sufficiently small.  At the final steps of our argument, we take $\eta \rightarrow 0$.
\end{defn}
We can easily extend the notion of entropy to smooth partitions (useful in the context of semiclassical analysis) by simply setting $\mu=0$ outside of $\mathcal{E}$:
\begin{defn} \label{defn:smooth_entropy_pressure}
Let $\kappa_{slab}>0$  be given.  Consider a corresponding smooth partition $\mathcal{P}_{\rm{sm}}$ near $\mathcal{E}$ and some weights $\bm{w}$.  We define the \textit{entropy of a smooth partition $\mathcal{P}_{\rm{sm}}$} as 
\begin{equation} \label{eqn:metric_ent}
H(\mu, \mathcal{P}_{\rm{sm}}) := - \sum_{k=1}^K \eta(\mu(\pi_k))
\end{equation}
and \textit{pressure of a (weighted) smooth partition $\mathcal{P}_{\rm{sm}}$}
\begin{equation*}
p(\mu, \mathcal{P}_{\rm{sm}}, \bm{w}) := - \sum_{k=1}^K \mu(\pi_k) \log \left( w_k^2 \, \mu(\pi_k) \right).
\end{equation*}
The corresponding definitions of $\mathcal{P}_{\rm{sm}}^{\vee n}$, $H_0^{n-1}$, and $p_0^{n-1}$ are clear.
\end{defn}

Thanks to this setup and assuming that $\mu$ does not charge the boundaries $\partial E_k$ for all $k=1,\dots, K$, we have that given any $\gamma'>0$ and $n \geq 1$, we can choose $\eta$ small enough and thus $\mathcal{P}_{\rm{sm}}$ such that
\begin{equation*} 
\left| H(\mu, \mathcal{P}^{\vee n}_{\rm sm}) - H(\mu, \mathcal{P}^{\vee n}) \right| \leq \gamma'.
\end{equation*}
See \cite[Section 3.1.3]{Non12} for more details on the compatibility of this assumption with semiclassical measures $\mu_{sc}$, itself based off \cite[Appendix A2]{An08} and \cite[Section 2.2.8]{AN07}.

\subsection{Quantum partitions}

\subsubsection{Quasiprojectors} \label{sect:quasiprojs}

With Definition \ref{def:smooth_partition} in hand, we form a notion of an approximate quantum partition of unity:
\begin{defn}
An \textit{approximate quantum partition of unity} is a collection $\mathcal{P}_{\rm{sm}, \rm{q}} = \{ \Pi_k := \Op_{\h}(\tilde{\pi}_k) \}_k$, whose individual elements $\Pi_k$ we call \textit{quasiprojectors} with corresponding symbols $\tilde{\pi}_k \in S^{-\infty, 0}(T^*M)$ that satisfy the following properties:
\begin{enumerate}
\item for each $k$, the symbol $\tilde{\pi}_k$ is constructed so that: $\supp \tilde{\pi}_k \subset \mathscr{O}_k$, $\sqrt{\pi}_k$ is the principal symbol, and $\Pi_k$ is selfadjoint.
\item the family of operators $\mathcal{P}_{\rm{sm}, \rm{q}}$ satisfies the relation
\begin{equation} \label{eqn:qou_prop}
\sum_{k=1}^K \Pi_k^2 = \Op_{\h}(\tilde{\chi}_{\kappa_{slab}}) + \mathcal{O}_K(\h^{\infty})
\end{equation}
where $\tilde{\chi}_{\kappa_{slab}} \in S^{-\infty, 0}(T^*M)$ itself satisfies
\begin{equation*}
\|\Op_{\h}(\tilde{\chi}_{\kappa_{slab}})\| = 1 + \mathcal{O}(\h^{\infty}), \, \, \supp \tilde{\chi}_{\kappa_{slab}} \subset \mathcal{E}_{\kappa_{slab}}, \, \mbox{ and } \tilde{\chi}_{\kappa_{slab}} = 1 \mbox{ on } \mathcal{E}_{\kappa_{slab}/2}.
\end{equation*}
\end{enumerate}
\end{defn}

Notice that $\Pi_k$ cannot simply be $\Op_{\h}(\sqrt{\pi_k})$ as we need lower order symbols in the expansion in order to guarantee property (\ref{eqn:qou_prop}); otherwise, our error term would be $\mathcal{O}(h)$.  By iteratively adjusting the lower-order symbols, \`a la standard procedures applied to $\tilde{\pi}_k$ and $\tilde{\chi}_{\epsilon_{slab}}$, we obtain the desired property.

\subsubsection{Refined quantum partitions}
Analogous to the refinements $\mathcal{P}^{\vee n}_{\rm sm}$ of an initial smooth partition $\mathcal{P}_{\rm sm}$, we consider their quantum versions.
\begin{defn}
Consider the Schr\"odinger propagator $U = e^{-iP(\h)/\h}$.  Then for a word $\bm{\alpha} \in \Sigma^n$, we define the refined (forward) quasiprojectors as 
\begin{equation*}
\Pi_{\bm{\alpha}} := U^{-(n-1)} \Pi_{\alpha_{n-1}} U \Pi_{\alpha_{n-2}} \dots \Pi_{\alpha_{1}} U \Pi_{\alpha_{0}}
\end{equation*}
The family $\mathcal{P}^{\vee n}_{sm,q} := \{ \Pi_{\bm{\alpha}} \}_{\bm{\alpha} \in \Sigma^n}$ is called the \textit{$n$-th refinement of the quantum partition $\mathcal{P}_{\rm{sm},\rm{q}}$}.  The definition for the refined backwards quasiprojectors uses $U^{-1}$ and is analogous.
\end{defn}

For words $\bm{\alpha}$ of length twice the so-called Ehrenfest time (\ref{eqn:Ehrenfest}), the quasiprojectors $\Pi_{\bm{\alpha}}$ fall outside of the scope of Egorov's Theorem for long times.  Despite this, our refined quasiprojectors continue to form quantum partitions of unity, as encompassed in the next proposition.  
\begin{prop}[Refined quasiprojectors form an approximate quantum partition of unity] (cf \cite[Prop 3.1]{Non12})\label{prop:pou_quasi}
Let $\delta_{energy}>0$ be given and choose $C>0$ such that  $C+1 < \delta_{energy}^{-1}$.  Let $N>0$ be given.  Then there exists a regime of $n$ and $\h$ such that the refined quantum partition $\mathcal{P}^{\vee n}_{sm,q}$ continues to form a quantum partition of unity microlocally near $\mathcal{E}$.  That is, for any symbol $\chi \in S^{-\infty, 0}(T^*X)$ supported inside of $\mathcal{E}_{\kappa_{slab}/2}$, we have
\begin{align*}
\, \, \, & \sum_{\bm{\alpha} \in \Sigma^n} \Pi_{\bm{\alpha}}^* \Pi_{\bm{\alpha}} = S_n, \, \, \, \hfill \|S_n\| = 1 + \mathcal{O}_{K}(\h^{N}) \\
& (Id - S_n) \, \Op_{\h}(\chi) = \mathcal{O}_{K}(\h^{N})
\end{align*}
for all $\h \leq \h_0(\kappa_{slab}, \delta_{energy}, K,N)$ and $1 \leq n \leq \lfloor C |\log \h| \rfloor$.
\end{prop}
In Sections \ref{sect:proj_micro_meas}, the constant $C$ above appears as $C_{\delta_{energy}}.$

\section{Ehrenfest time and the hyperbolic dispersive estimate} \label{sect:hyp_disp}

\begin{defn} \label{defn:Ehrenfest}
Let $\delta_{Ehr} \in (0,1), \kappa_{slab}>0,$ and $\h>0$ be given.  We call
\begin{equation} \label{eqn:Ehrenfest}
T_{\delta_{Ehr}, \kappa_{slab}, \h} := \frac{(1-\delta_{Ehr})}{2 \lambda_{\max}(\kappa_{slab})} |\log \h|
\end{equation}
the \textit{Ehrenfest time} on the energy neighbourhood $\mathcal{E}_{\kappa_{slab}}$.  The number $\lambda_{\max}(\kappa_{slab})$ is the maximal expansion rate across the entire slab $\mathcal{E}_{\kappa_{slab}}$.  Note that for $\kappa_{slab}$ small enough (depending on $M$), $\lambda_{\max}(\kappa_{slab}) = \lambda_{\max} + \mathcal{O}(\kappa_{slab}).$
\end{defn}

\begin{defn}[Coarse-grained unstable Jacobian]
Let $\mathcal{P}$ be a partition of $\mathcal{E}$.  Recall
\begin{equation*}
J^u(\rho) : = |\det D(g^1)_{\upharpoonright E^u(\rho) }|
\end{equation*}
is the \textit{unstable Jacobian for $g^1$ at $\rho \in \mathcal{E}$} with $E^u(\rho)$ being the unstable subspace at $\rho \in \mathcal{E}$.  For $E_{\alpha_j} \in \mathcal{P}$, we define $J^u(\alpha_j) := \inf \{ J^u(\rho) \, | \, \rho \in E_{\alpha_j} \}$.  We call the quantity
\begin{equation}\label{eqn:coarse_grain_Jac}
J_n^u(\bm{\alpha}) := \prod_{i=0}^{n-1} J^u(\alpha_i)
\end{equation}
the \textit{coarse-grained unstable Jacobian at $\bm{\alpha} \in \Sigma^n$}.
\end{defn}

The following is from \cite[Corollary 3.7]{Non12} and \cite[Theorem 2.7]{AN07}:
\begin{prop}[Hyperbolic dispersive estimate] \label{prop:hyp_disp}
Let $\delta_{Ehr} \in (0,1)$ and $\kappa_{slab}>0$ be given. Choose $\delta_{energy}>0$ small enough with $\frac{4}{\lambda_{\max}} + 1 < \delta_{energy}^{-1} $.  Furthermore, consider a smooth partition $\mathcal{P}_{\rm{sm}}$ near $\mathcal{E}$.  Then, there exists $\h_{0}(\mathcal{P}_{\rm{sm}}, \delta_{energy}, \delta_{Ehr}, \kappa_{slab})>0$, $C>0$, and $c_0>0$ such that for any $0 < \h \leq \h_{0}$, any integer $n \in [0, 2 T_{\delta_{Ehr}, \kappa_{slab}, \h}]$ and any two sequences $\bm{\alpha}, \bm{\beta}$ of length $n$, the following estimate holds:
\begin{equation} \label{eqn:hyp_disp}
\left\| \Pi_{\bm{\alpha}} U^n \, \Pi_{\bm{\beta}} U^{-n} \chi^{(n)}\left( P(\h)-\frac{1}{2} \right) \right\| \leq C \,  h^{-(d-1+c_0\delta_{energy})/2} J_n^u(\bm{\alpha})^{-1/2} \,  J_n^u(\bm{\beta})^{-1/2}
\end{equation}
\end{prop}
The entropic uncertainty principle (described in Section \ref{sect:EUP}) requires that we estimate the operator norm $\Pi_{\bm{\alpha}} \, U^n \, \Pi_{\bm{\beta}} \, U^{-n} \chi^{(n)}\left(P(\h) - \frac{1}{2} \right)$, resulting from our choice of partitions $\bm{\rho}$ and $\bm{\tau}$, therefore feeding the right-hand side of (\ref{eqn:hyp_disp}) into the crucial quantity $c_{cone}$.  In fact, \cite[Section 3.2]{DJ18} and \cite[Theorem 7.1]{Riv10} show this operator is actually pseudodifferential.

\end{document}